\documentclass{article}

\PassOptionsToPackage{round,compress}{natbib}
\usepackage[preprint]{neurips_2020}

\usepackage[utf8]{inputenc} 
\usepackage[T1]{fontenc}    
\usepackage{hyperref}       
\usepackage{url}            
\usepackage{booktabs}       
\usepackage{amsfonts}       
\usepackage{nicefrac}       
\usepackage{microtype}      

\usepackage{amsmath,amsfonts,amssymb,amsthm,mathtools}
\usepackage[linesnumbered,ruled,vlined]{algorithm2e}
\usepackage{cleveref}
\usepackage{overpic}
\usepackage{subcaption}
\usepackage{wrapfig}
\usepackage{enumitem}
\setlist{nosep}

\SetCommentSty{mycommfont}
\usepackage{siunitx}
\hypersetup{
    colorlinks=true,
    linkcolor=blue,
    citecolor = blue
}

\usepackage{xcolor}
\usepackage{hyperref}
\usepackage{physics}
\newtheorem{assumption}{Assumption}
\newtheorem{lemma}{Lemma}
\newtheorem{theorem}{Theorem}

\newtheorem{example}{Example}

\newtheorem{proposition}{Proposition}
\let\tr\undefined
\DeclareMathOperator{\tr}{Tr}
\DeclareMathOperator{\E}{\mathbb{E}}
\DeclareMathOperator{\R}{\mathbb{R}}
\DeclareMathOperator{\Ball}{\mathrm{Ball}}
\DeclareMathOperator{\Sphere}{\mathrm{Sphere}}

\newcommand{\cst}{c_0}
\newcommand{\sr}{\rho_0}

\newcommand{\lopt}{{K^*_{\mathrm{lin}}}}
\newcommand{\clin}{c_{\mathrm{lin}}}
\newcommand{\rholin}{\rho_{\mathrm{lin}}}
\newcommand{\abk}{\Gamma}

\newboolean{isfullversion}
\setboolean{isfullversion}{true}
\newcommand{\fvtest}[2]{\ifthenelse{\boolean{isfullversion}}{#1}{#2}}

\newboolean{showcomments}
\setboolean{showcomments}{false}
\newcommand{\guannan}[1]{\ifthenelse{\boolean{showcomments}}
{ \textcolor{red}{(Guannan:  #1)}}{}}
\newcommand{\chenkai}[1]{\ifthenelse{\boolean{showcomments}}
{ \textcolor{cyan}{(Chenkai:  #1)}}{}}
\newcommand{\slow}[1]{\ifthenelse{\boolean{showcomments}}
{ \textcolor{brown}{(Steven:  #1)}}{}}
\newcommand{\adam}[1]{\ifthenelse{\boolean{showcomments}}
{ \textcolor{green}{(Adam:  #1)}}{}}

\allowdisplaybreaks

\title{Combining Model-Based and Model-Free Methods for Nonlinear Control: A~Provably Convergent Policy Gradient Approach}

\author{%
  Guannan Qu\\
  Caltech, Pasadena, CA\\
  \texttt{gqu@caltech.edu} 
  \And 
  Chenkai Yu\\
  Tsinghua University, Beijing\\
  \texttt{yck17@mails.tsinghua.edu.cn} 
  \And
  Steven Low\\
  Caltech, Pasadena, CA\\
  \texttt{slow@caltech.edu} 
  \And
  Adam Wierman\\
  Caltech, Pasadena, CA\\
  \texttt{adamw@caltech.edu} 
}

\begin{document}

\maketitle

\begin{abstract}
Model-free learning-based control methods have seen great success recently. However, such methods typically suffer from poor sample complexity and limited convergence guarantees. This is in sharp contrast to classical model-based control, which has a rich theory but typically requires strong modeling assumptions. In this paper, we combine the two approaches to achieve the best of both worlds. We consider a dynamical system with both linear and non-linear components and develop a novel approach to use the linear model to define a warm start for a model-free, policy gradient method. We show this hybrid approach outperforms the model-based controller while avoiding the convergence issues associated with model-free approaches via both numerical experiments and theoretical analyses, in which we derive sufficient conditions on the non-linear component such that our approach is guaranteed to converge to the (nearly) global optimal controller. 
\end{abstract}

\section{Introduction}
Recent years have seen great success in using learning-based methods for the control of dynamical systems. Examples cut across a broad spectrum of applications, including robotics \citep{levine2015endtoend,duan2016benchmarking}, autonomous driving \citep{li2019reinforcement}, energy systems \citep{wu2020deep}, and more. Many of these learning-based methods are model-free in nature, meaning that they do not explicitly estimate the underlying model and do not explicitly make any assumptions on the parametric form of the underlying model \citep{sutton1988learning,bertsekas2011dynamic,williams1992simple}. 
Examples of such methods include policy gradient methods \citep{fazel2018global,bu2019lqr,li2019distributed} and approximate dynamic programming \citep{bradtke1994adaptive,tu2017least,krauth2019finite}. 
Because model-free methods do not explicitly assume a parametric model class, they can potentially capture hard-to-model dynamics \citep{clavera2018modelbased}, which has led to empirical success in highly complex tasks \citep{levine2013guided,salimans2017evolution,recht2019tour}. 
However, the theoretic understanding of model-free approaches is extremely limited, and empirically they suffer from poor sample complexity and convergence issues \citep{nagabandi2018neural,tu2018gap}. 

This stands in contrast to the classical model-based control, where one first estimates a parametric form of the model (e.g. linear state space model) and then develops a controller using tools from classic control theory. This approach has a rich history, including theoretical guarantees \citep{zhou1996robust,dean2017sample}, and is typically more sample efficient \citep{tu2018gap}. However, one major drawback of model-based control is that the model class might fail to capture complex real-world dynamics, in which case model error makes theoretical guarantees invalid. 

Given the contrasts between model-free control and model-based control, the literature that focuses on providing a theoretic understanding of the two approaches is largely distinct, with papers focusing on either model-based approaches (e.g. \citet{dean2019sample,mania2019certainty,simchowitz2020naive,simchowitz2020improper} or model-free approaches (e.g. \citet{fazel2018global,malik2018derivative,bu2019lqr}). 
There have been recent empirical approaches suggesting that model-based and model-free approaches can be combined to achieve the benefit of both, e.g., \cite{nagabandi2018neural,silver2018residual,clavera2018model}; however, a theoretical understanding of the interplay between the approaches, especially when the dynamical system is nonlinear, remains open. Thus, in this paper we ask the following question:

{\begin{center}\textit{Can model-based and model-free methods be combined to provably achieve the benefits of both?}\end{center} }

\textbf{Contribution. } In this paper, we answer the question in the affirmative in the context of a non-linear control model.  Specifically, we consider a dynamical system whose state space representation is a sum of two parts: a linear part, which is the most commonly used model class in model-based control, and a non-parametric non-linear part.  This form of decomposition is widely used in practice. For example, engineers often have good approximate linear models for real-world dynamical systems such as energy systems \citep{benchaib2015acpower} and mechanical systems \citep{magdy2019modeling}.  The difference between the linear approximation and the real dynamics is often nonlinear and nonparametric, though understood to be small. 

In this context, we introduce an approach for combining model-based methods for the linear part of the system and model-free approaches for the nonlinear part. In detail, we first use a model-based approach to design a state-feedback controller based on the linear part of the model. Then, we use this controller to warm start a model-free policy search. This warm start is similar in spirit to several empirically successful methods in the recent literature, e.g. \cite{nagabandi2018neural,silver2018residual}, however, no theoretical guarantees are known for existing approaches. In contrast, we prove guarantees on the convergence of the approach to an (almost) globally optimal state-feedback linear controller.  Our analysis shows that the approach combines the benefits of model-based methods and model-free methods, capturing the unmodeled dynamics ignored by the model-based control while avoiding the convergence issues often associated with model-free approaches.

The key technical contribution underlying our approach is a landscape analysis of the cost as a function of the state-feedback controller. We show that the model-based controller obtained from the linear part of the system falls inside a convex region of the cost function which also contains the (almost) global minimizer. As a result, when using a warm start from the model-based controller, our approach is guaranteed to converge to the global minimizer. To highlight the necessity of the warm start, we show examples in which the landscape is non-convex and contains spurious local minima and even has a disconnected domain. Thus, a model-free approach that ignores model information completely may fail to converge to the global minimizer. To the best of our knowledge, ours is the first result to provide a theoretical understanding of the landscape for model-free policy search in non-linear control. 


\textbf{Related Work.} Our work is mostly related to the class of model-free policy search methods for the Linear Quadratic Regulator (LQR), e.g. zeroth order policy search in \cite{fazel2018global,malik2018derivative,bu2019lqr,mohammadi2019convergence,li2019distributed} and actor-critic methods in \citet{yang2019global}. A common theme in this line of work is that the underlying dynamical system is assumed to be \emph{linear}, under which the cost function is shown to satisfy a ``gradient dominance'' property \citep{fazel2018global}, which implies the model-free policy search method will converge to the global optimal controller. While these results provide a theoretic understanding of model-free methods, the benefits of using model-free methods for linear systems is not clear.  For example,  \citet{tu2018gap} shows that when the dynamics is actually linear, model-based methods are more sample efficient than model-free approaches. On the other hand, applications where model-free approaches have seen the most success are those involving the control of nonlinear dynamics \citep{pong2018temporal}. However, though there has been empirical success, an understanding of model-free approaches for nonlinear systems is lacking.  Our work makes an initial step by, for the first time, analyzing a model-free policy search method for nonlinear systems (with a particular structure).

Our work is also related to empirical approaches suggested in the literature on reinforcement learning that involve augmenting model-free reinforcement learning with model-based approaches for various goals  \citep{che2018combining,vuong2019uncertainty,pong2018temporal}, such as for gradient computation \citep{mishra2017prediction,heess2015learning}, generate trajectories for model-free training \citep{gu2016continuous,weber2017imaginationaugmented,feinberg2018model}. Among these, the most related to our work are \cite{bansal2017mbmf,nagabandi2018neural,silver2018residual,johannink2019residual}, which use model-based methods as a starting point for model-free policy search. 
However, these papers focus on empirical evaluation, and to the best of our knowledge, we are the first to provide a theoretic justification on the effectiveness of combining model-based and model-free methods. 

Beyond the above, our work is also related to a variety of areas at the interface of learning and control:

\textit{Model-based LQR.} When the model is linear and is known, the optimal control problem can be solved via approaches like Algebraic Ricatti Equation \citep{zhou1996robust} and dynamic programming \citep{bertsekas2005dp_vol1}. When the linear model has unknown parameters, various system identification approaches have been proposed to estimate the system parameter, e.g. classic results such as \cite{ljung1999system,lennart1999system} or more recent ones with a focus on finite sample complexity, e.g.,  \cite{simchowitz2018learning,oymak2019non,sarkar2019finite}. 
In addition, there have been recent efforts to provide end-to-end frameworks that combine system identification and control design  \citep{dean2019sample}, sometimes in an online setting, e.g. \citet{abbasi2011regret,faradonbeh2017finite,ouyang2017learning,dean2018regret,cohen2019learning,mania2019certainty,simchowitz2020naive,simchowitz2020improper}.   

\textit{Control of nonlinear systems.} There is a vast literature on the control of nonlinear dynamical systems, see e.g. \citep{slotine1991applied,isidori2014nonlinear}, including techniques like feedback linearization \citep{westenbroek2019feedback}. Specifically, our model is related to a practice in nonlinear control where one first linearizes the nonlinear system and design a controller based on the linear model \citep[Sec. 3.3]{slotine1991applied}. Our proposed approach goes beyond this by using model-free policy search to improve the controller obtained from the linear system.

\textit{Robust control.} The fact that our model is a summation of a linear part and a small nonlinear part can be understood from the robust control angle \citep{dullerud2013course}, where the linear model can be viewed as the nominal plant and the nonlinear part can be viewed as an uncertain perturbation \citep{petersen2014robust}. However, robust control seeks to design controllers with worst-case guarantees against all possible perturbations \citep[Sec 4]{doyle2013feedback}, whereas our work seeks to learn the best controller for the actual instance of the perturbation (the non-linear part of the model).

\textit{Other related work. } 
There have also been other approaches that involve the decomposition of the model into a known model-based part and an unknown part.  Most notable are \cite{koller2018learning,Shi_2019}, though the model and focus therein are very different from ours.


\section{Proposed Framework}
We consider a dynamical system with state $x_t\in\mathbb{R}^n$ and control input $u_t\in \mathbb{R}^p$, 
\begin{equation}
  x_{t+1} = A x_t + B u_t + f(x_t), \label{eq:system}
\end{equation}
where $A$ is $n$-by-$n$, $B$ is $n$-by-$p$, and $f: \R^{n}\rightarrow\R^n$ satisfies $f(0) = 0$ and is ``small'' compared to $A$ and $B$. We focus on the class of linear controllers, $u_t = -Kx_t$ for $K\in\R^{p\times n}$ and we consider the following quadratic cost function $C:\R^{p\times n} \rightarrow \R$,
\begin{equation}
  C(K) = \mathbb{E}_K \sum_{t=0}^\infty \big( x_t^\top Q x_t + u_t^\top R u_t \big), \label{eq:cost}
\end{equation}
where the expectation is taken with respect to $x_0$ that is drawn from a fixed initial state distribution $\mathcal{D}$, and the subscript $K$ in the expectation indicates the trajectory $\{x_t\}_{t=0}^\infty $ in the expectation is generated by controller $K$.

The system in \eqref{eq:system} is the sum of a linear part and a ``small'' non-linear part $f$. Such a decomposition can be found in many practical situations, as discussed in the introduction.  
In such settings, $f$ represents the error of the approximated linear model, which is small if the linear approximations are accurate in practice. Alternatively, \eqref{eq:system} can be a result of linearization of a non-linear model, with $f$ capturing the higher order residuals. 

To simplify exposition, we assume $A,B$ are known whereas $f$ is unknown since, in various engineering domains, the approximated linear model ($A,B$) is readily available. When $(A,B)$ are not known, they can also be estimated from data by using standard least square techniques. 

\textbf{Combining model-based and model-free control.} We propose a framework that combines model-based and model-free methods to find an 
optimal linear state feedback controller that minimizes the cost \eqref{eq:cost}. Concretely, the framework works as follows: 
\begin{itemize}
    \item Compute model-based controller $\lopt$ to be the optimal LQR controller for linear system $(A,B)$ and cost matrices $(Q,R)$.
    \item Use $\lopt$ as an initial point for model-free policy search. There can be many variants of policy search, including zeroth-order policy search \citep{fazel2018global} or actor-critic methods \citep{yang2019global}. In Section~\ref{sec:mainresult} we propose a concrete approach (Algorithm~\ref{algo:policysearch}).
\end{itemize}

Compared with a standard model-free approach, where the initial point is unspecified and is usually obtained through trial and error, this hybrid approach makes use of model-based control to warm start the model-free policy search algorithm. This intuitive idea is powerful given the complexity of the cost landscape. To illustrate the importance of this warm start approach, we provide two examples (Examples~\ref{example:1} and \ref{example:2}) showing that, even when $f$ is small compared to $A,B$, the landscape of $C(K)$ may contain spurious local minima (Example~\ref{example:1}), and the set of stabilizing state feedback controllers may not even be connected (Example~\ref{example:2}). As such, model-free approaches will likely fail to converge to the global minimizer.  In contrast, in the examples, the model-based controller $\lopt$ stays within the attraction basin of the global minimizer, and hence the proposed hybrid approach with the model-based warm start converges. In the next section, we formalize this intuition and provide theoretic results on the landscape of $C(K)$ as well as the convergence of the proposed approach.

\begin{example}[Cost landscape may contain spurious local minima] \label{example:1} Consider the following one-dimensional dynamics $x_{t+1} = 0.5 x_t + u_t + f(x_t)$, and $f(x) = 0.01 x/(1 + 0.9 \sin(x))$, satisfying $|f(x)|\leq 0.1 |x| $. We set $x_0 = 50, Q = 10, R=1$. When using a linear state feedback controller $u_t = -K x_t$, the cost is given in Figure~\ref{fig:counter_example}, which has many local minima. However, $\lopt$ lies within the attraction basin of the global minimizer $K^*$ and is in fact very close to $K^*$.
\end{example}

\begin{figure}
  \begin{minipage}[b]{0.49\textwidth}
    \centering
    \includegraphics[width = .97\textwidth]{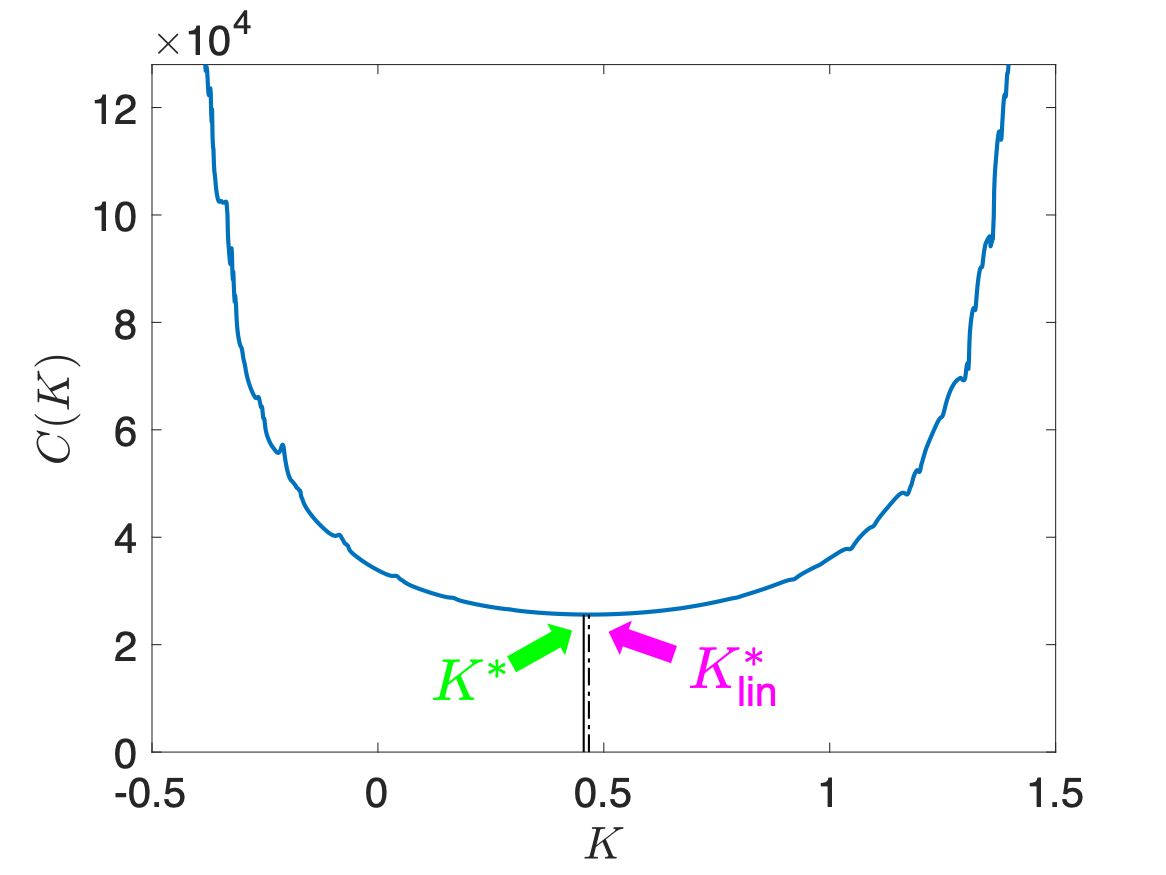}
    \caption{Cost landscape in Example \ref{example:1}.}
    \label{fig:counter_example}
  \end{minipage} \hfill
  \begin{minipage}[b]{0.49\textwidth}
    \centering
    \begin{overpic}[width = .96\textwidth]{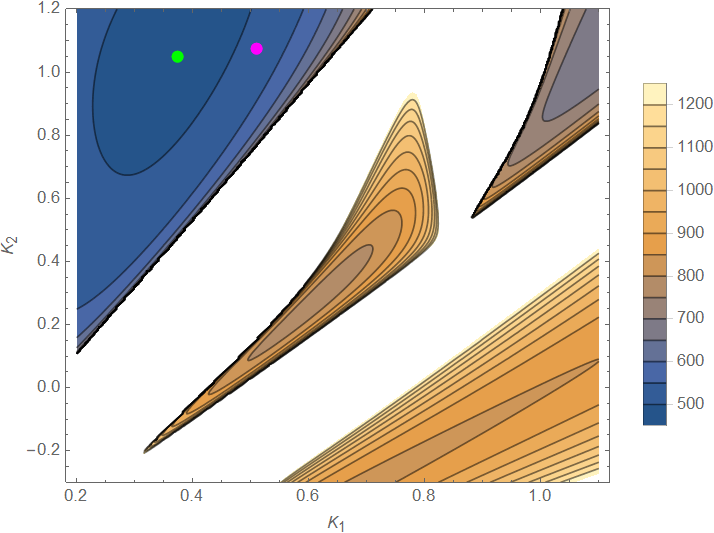}
      \definecolor{Magenta}{rgb}{1,0,1}
      \put (17,65) {\small \color{green} $K^*$}
      \put (31,62.5) {\small \color{Magenta} $\lopt$}
    \end{overpic}
    \caption{Cost landscape in \Cref{example:2}. 
    }
    \label{fig:disconnected}
  \end{minipage}
\end{figure}

\begin{example}[Finite-cost controllers may be disconnected] \label{example:2}
  Suppose $n = 2$ and $p = 1$. Let
  {\small\[A = 0.95 \begin{bmatrix} \cos 0.2 & - \sin 0.2 \\ \sin 0.2 & \cos 0.2 \end{bmatrix},~
  B = \begin{bmatrix} 0.2 \\ 0.15 \end{bmatrix},~
  Q = \begin{bmatrix} 1 & -0.999 \\ -0.999 & 1 \end{bmatrix},~
  R = 0.5,~
  x_0 = \begin{bmatrix} 5 \\ -6 \end{bmatrix},\]
  \begin{align*}
    f(x)
    & = \frac{0.1 (\begin{bsmallmatrix} 3 \\ 0 \end{bsmallmatrix} - x) - (A - I) x - B \begin{bmatrix} -1 & -0.2 \end{bmatrix} x}{(\|x - \begin{bsmallmatrix} 3 \\ 0 \end{bsmallmatrix}\|^2 + 1)^2} + \frac{0.7 (\begin{bsmallmatrix} 4.5 \\ -3 \end{bsmallmatrix} - x) - (A - I) x - B \begin{bmatrix} 0 & 0.7 \end{bmatrix} x}{(\|x - \begin{bsmallmatrix} 4.5 \\ -3 \end{bsmallmatrix}\|^2 + 1)^2} \\
    & \quad + \frac{0.9 (\begin{bsmallmatrix} 5 \\ -1 \end{bsmallmatrix} - x) - (A - I) x - B \begin{bmatrix} -0.2 & 0.5 \end{bmatrix} x}{(\|x - \begin{bsmallmatrix} 5 \\ -1 \end{bsmallmatrix}\|^2 + 1)^2} + f_0,
  \end{align*}}where $f_0$ is such that $f(\begin{bsmallmatrix} 0 \\ 0 \end{bsmallmatrix}) = \begin{bsmallmatrix} 0 \\ 0 \end{bsmallmatrix}$. In this case, the set of controllers $K = [K_1, K_2]$ with finite cost is not connected, as shown in \Cref{fig:disconnected}.
  Moreover, this phenomenon exists even for very small $f$. Starting from the above values, we can simultaneously make $A$ closer to $I$, $B$ closer to 0, and the coefficients $0.1, 0.7, 0.9$ in function $f$ closer to 0 (with the same factor) in order to maintain this phenomenon. A detailed explanation of this example can be found in \fvtest{\Cref{sec:explain}}{Appendix~A} in the supplementary material.
\end{example}




\textbf{Notation.} In this paper, $\Vert\cdot\Vert$ is always the Euclidean norm for vectors and the spectrum norm for matrices. $\Vert\cdot\Vert_F$ is the Frobenius norm. For matrices $A,B$ of the same dimension, $\langle A,B\rangle = \tr(A^\top B)$ is the trace inner product. Additionally, $y_1 \lesssim y_2$ and $y_1 \asymp y_2$ mean $y_1 \leq c y_2$ and $y_1 = c y_2$ respectively for some numerical constant $c$.

\section{Main Results}\label{sec:mainresult}
Our main technical result characterizes the landscape of the cost function in order to prove the convergence of the proposed approach combining model-based and model-free techniques. For concreteness, we use a particular instance of the policy search method and show its convergence, but the approach is more general and can be extended to other methods. 

Before stating our results, we discuss their assumptions.  The first assumption is about the pair $Q,R$ in the cost function and is standard \citep{mania2019certainty}. 

\begin{assumption}\label{assump:QR} $Q$ and $R$ are positive definite matrices satisfying $R + B^\top Q B \succeq \sigma I$ for some $\sigma > 0$, and $\Vert Q\Vert\leq 1$, $\Vert R\Vert\leq 1$.
\end{assumption}
The assumption $\Vert Q\Vert\leq 1, \Vert R\Vert \leq 1$ in Assumption~\ref{assump:QR} is for ease of calculation, and is without loss of generality as we can always rescale the cost function to guarantee it is satisfied. Our next assumption concerns the pair $(A,B)$ and is again standard \citep{dean2017sample}. 

\begin{assumption}\label{assump:AB}
The pair $(A,B)$ is controllable; $\lopt$ is the optimal controller associated with the linear system $x_{t+1} = A x_t + B u_t$. Also let $\Vert (A-B\lopt)^t\Vert\leq \clin \rholin^t$, for some $\rholin\in(0,1)$ and $\clin>0$. Further, denote $\Gamma = \max(\Vert A\Vert, \Vert B\Vert, \Vert \lopt\Vert,1)$.
\end{assumption}

Next, we make an assumption on the initial state distribution. 

\begin{assumption}\label{assump:initial}
The initial state distribution $\mathcal{D}$ is supported in a region with radious $D_0$. Further, $\E x_0 x_0^\top \succeq \sigma_{x} I$ for some $\sigma_x >0$.
\end{assumption}

The requirement of bounded support is only for simplification of the proof. It can be replaced with a bound on the second and the third moment of the initial state if desired at the expense of extra complexity. Finally, we assume that $f$ and the Jacobian of $f$ are Lipschitz continuous or, in other words, the first and second order derivatives of $f$ are bounded. This quantifies the ``smallness'' of $f$.

\begin{assumption}\label{assump:ell}
We assume $f$ is differentiable, $f(0)=0$, $\Vert f(x) - f(x') \Vert \leq \ell \Vert x - x'\Vert$, and $\Vert \frac{\partial f (x)}{\partial x} - \frac{\partial f(x')}{\partial x}\Vert \leq \ell' \Vert x - x'\Vert$ for some $\ell,\ell'>0$, where $\frac{\partial f (x)}{\partial x}$ is the Jacobian of $f(x)$ w.r.t.\ $x$. 
\end{assumption}

Before we state our result, we must also define what we mean by the ``global'' domain of $C(K)$. One natural definition for the domain of $C$ is the set of (global or local) stabilizing controllers for the nonlinear system~\eqref{eq:system}. However, to the best of our knowledge, the stabilization of nonlinear systems is a challenging topic and such a set is not clearly characterized. For this reason, we consider an alternative domain $\Omega(\cst,\sr) = \{K: \Vert (A-BK)^t\Vert\leq \cst\sr^t\}$ for some $\cst\geq 1, \sr \in(0,1)$ to be chosen later. We consider this domain since it is clearly characterized and also because when $\sr\rightarrow 1$, $\cst\rightarrow\infty$, this set captures the set of almost all stabilizing controllers for the linear system $(A,B)$. 

We now move to our results.  Our first result characterizes the landscape of the cost function. It shows that when $\ell$ and  $\ell'$ (the Lipschitz constant for $f$ and Jacobian of $f$ respectively) are small enough, $C(K)$ achieves its global minimum inside a local neighborhood of $\lopt$, the optimal controller for the linear part of the system. Further, within this local neighborhood, $C(K)$ is strongly convex and smooth.  Theorem~\ref{thm:landscape} is our most technical result and a proof is provided in \fvtest{\Cref{sec:landscape}}{Appendix~B} in the supplementary material.

\begin{theorem} \label{thm:landscape}
  For any $\sr\in [\frac{\rholin + 1}{2},1)$ and $\cst\geq 2 \clin$, let $\Omega = \Omega(\cst,\sr)$.
  If $\ell \lesssim  \frac{ (\sigma \sigma_x)^2 (1-\sr)^8}{ \abk^9 \cst^{15} D_0^4}$, $\ell' \lesssim \frac{ (\sigma \sigma_x)^2 (1-\sr)^8}{ \abk^9 \cst^{18} D_0^5}$, then:
  \begin{enumerate}[label=(\alph*)]
    \item $C(K)$ is finite in $\Omega$ and the trajectories satisfies $\Vert x_t\Vert \leq 2\cst (\frac{\sr+1}{2})^t\Vert x_0\Vert$ for any $x_0\in\R^n, K\in\Omega$;
    \item there exists a region $\Lambda(\delta)=\{K: \Vert K- \lopt\Vert_F\leq \delta \} \subset\Omega $ with $\delta  \asymp \frac{\sigma_x\sigma (1-\sr)^4}{ \abk^5 \cst^7 D_0^2}$ such that  $C(K)$ is $\mu$-strongly convex and $h$-smooth inside $\Lambda(\delta)$, with $\mu = \sigma \sigma_x$ and $h \asymp \frac{\abk^4 \cst^4 D_0^2}{(1-\sr)^2}$;
    \item the global minimum of $C(K)$ over $\Omega$ is achieved at a point $K^*\in \Lambda(\frac{\delta}{3}) $, which is also the unique stationary point of $C(K)$ inside $\Lambda(\delta)$.
  \end{enumerate}
\end{theorem} 

We comment that, while our landscape result is a local convexity result around the global minimum $K^*$, we are also able to show that $\lopt$ (which can be computed efficiently) is within the convex region around  $K^*$ and, as such, within the attraction basin of $K^*$. This is different than existing landscape analysis for non-convex optimization in other contexts like deep learning, where only local convexity is shown without showing how to enter its attraction basin \citep{pmlr-v97-oymak19a,azizan2019stochastic}. 

Given the landscape result, it is perhaps not surprising that the model-free policy search method converges to the global minimizer $K^*$ when warm starting with the model-based optimal LQR controller $\lopt$, because both $K^*$ and $\lopt$ lie in the same convex region of the cost function. In the following, we prove this formally by considering a version of model-free policy search algorithm - the zeroth order policy search with one point gradient estimator. The proposed algorithm is stated in Algorithm~\ref{algo:policysearch} with the gradient estimator subroutine given in Algorithm~\ref{algo:gradientestimator}. Our result, Theorem~\ref{thm:policygradient}, shows that the landscape result in Theorem~\ref{thm:landscape} ensures that Algorithm~\ref{algo:policysearch} converges to the global minimum of $C(K)$ over $\Omega$, hence outperforming the model-based controller and avoiding the non-convergence issue of model-free approaches shown before. 

\begin{theorem} \label{thm:policygradient}
  Under the conditions in Theorem~\ref{thm:landscape}, for any $\varepsilon>0$ and $\nu \in(0,1)$, if the step size $\eta \le \frac{1}{h}$, the number of gradient descent steps $M\geq \frac{1}{\eta\mu} \log(\delta \sqrt{h / \varepsilon})$, and the gradient estimator parameters satisfy $r \le \min(\frac{1}{3}\delta, \frac{1}{3h}e_{grad})$, 
  {\small \[J \ge \frac{1}{e_{grad}^2} \frac{d^3}{r^2} \log\frac{4d M }{\nu} \max(18(C(K^*) + 2 h \delta^2)^2 , 72 C_{\max}^2), \quad T\geq \frac{2}{1-\sr} \log \frac{6 d C_{\max}}{e_{grad} r},\]}where $e_{grad} = \min(\frac{\mu}{2} \sqrt{\frac{\varepsilon}{h}}, \mu\frac{\delta}{3})$, $d = pn$, and $C_{\max} = \frac{40 \abk^2 \cst^2}{1-\sr} D_0^2$, then with probability at least $1-\nu$, $C(K_{M}) - C(K^*) \leq  \varepsilon$. 
\end{theorem}

A proof of Theorem~\ref{thm:policygradient} is provided in \fvtest{\Cref{sec:policy_grad}}{Appendix C} in the supplementary material. In addition to the above theoretical guarantees, the hybrid approach numerically appears to have better sample complexity even when the model-free methods do converge. We illustrate such results in the next section.

Our result shows that the proposed hybrid approach is guaranteed to converge to the global minimum only when $\ell,\ell'$ are bounded. Such a requirement on $\ell,\ell'$ is intuitive since when the ``size'' of $f$ is much larger than the linear part $(A,B)$, a warm start based on the linear model does not make much sense as the linear model is a poor estimation of the dynamics. There should be a threshold on the ``size'' of $f$, below which the hybrid approach will work. Our result provides a (potentially conservative) lower bound on the threshold. Tighter bounds are interesting goals for future work.


Finally, we comment that Algorithm~\ref{algo:policysearch} with the gradient estimator Algorithm~\ref{algo:gradientestimator} is but one of many possibilities for policy search methods.  There are various results suggesting ways to reduce the variance of the gradient estimator \citep{greensmith2004variance,nesterov2017random,preiss2019analyzing} that could also be incorporated into the framework here.



\begin{algorithm}\caption{Model-Free Policy Search with Model-Based Warm Start}\label{algo:policysearch}
\DontPrintSemicolon
\KwIn{Linear Model $(A,B)$, cost matrix $(Q,R)$, parameters $\eta,M,r,J,T$ }

 $\lopt\gets \textrm{OPT-LQR}(A,B,Q,R)$  \tcp*{Find the optimal controller for the linear system}
 $K_0 \gets \lopt$ \tcp*{Warm start}
\For{$m=0,1,\ldots,M-1$}{
     $\widehat{\nabla C}(K_m) \gets \textrm{GradientEstimator}(K_m,r,J,T)$\;
     $K_{m+1} \gets K_m - \eta \widehat{\nabla C}(K_m)$\;
}
\Return{$K_{M}$}\;
\end{algorithm}

\begin{algorithm}\caption{GradientEstimator}\label{algo:gradientestimator}
\DontPrintSemicolon
\KwIn{Controller $K$, parameters $r,J,T$}
\For{$j=1,2,\ldots,J$}{
\tcc{Sample random direction $U_j$ from sphere with radius $r$ in Frobenius norm}

Sample $ U_j \sim \Sphere(r)$ \;

\tcc{Sample a trajectory under perturbed controller $K + U_j$}
Sample $x_0\sim \mathcal{D}$\; 
\For{$t=0,1,\ldots,T$}{ 
    Set $u_t = -(K+  U_j) x_t$\;
    Receive the next point $x_{t+1}$ from system\;
}
Calculate approximate cost $\widehat{C}_j = \sum_{t=0}^T [x_t^\top Q x_t + u_t^\top R u_t ]$\;
}
\Return{$\widehat{\nabla C}(K) = \frac{1}{J}\sum_{j=1}^J \frac{d}{r^2}\widehat{C}_j  U_j$ where $d = pn$} \tcp*{One point gradient estimator}
\end{algorithm}

\section{Numerical Experiments}
To illustrate our approach, we contrast it with model-free and model-based approaches using two sets of experiments: (i) synthetic random instances and (ii) the cart inverted pendulum.

\subsection{Synthetic experiments}
Our first set of experiments focuses on random synthetic examples.  We set $n$ (the dimension of state) and $p$ (the dimension of input) to be $2$. We generate $A$ and $B$ randomly, with each entry drawn from a Gaussian distribution $N(0,1)$, where $A$ is normalized so that the spectral radius of $A$ is $0.5$. The initial state distribution $\mathcal{D}$ is a uniform distribution over a fixed set of $2$ initial states, which are drawn from  i.i.d.\ zero-mean Gaussians with norm normalized to be $2$. The cost is set as $Q = 2I, R=I$.
We set $f(x) = \ell x/(1 - 0.9 \sin(x))$, where all operations here are understood as entry wise and $\ell$ is a parameter that we increase from $\ell =0.005$ to $\ell = 0.08$. For each $\ell$, we run both our hybrid approach and the model-free approach (starting from $K=0$ as this system is open loop stable) with algorithm parameters $\eta = 0.01$, $T = 50$, $r = 0.001$, $J = 10$, and $M = 200$. We repeat the above procedures for $50$ times, each time with $A,B$ and $\mathcal{D}$ regenerated, and then plot the final cost achieved by both approaches  (normalized as the improvement over the model-based LQR controller)\footnote{The improvement is counted as $-\infty$ if a run fails to converge to a stabilizing controller. } as a function of $\ell$ in Figure~\ref{fig:toy_finalcost}. We also plot the sample complexity as a function of $\ell$ for both approaches in Figure~\ref{fig:toy_complexity}, where sample complexity is the number of state samples needed for the respective algorithm to converge.\footnote{The sample complexity is counted as $\infty$ if a run doesn't converge to a stabilizing controller.} The results show that both the proposed hybrid approach and the model-free approach can outperform the model-based LQR controller. Moreover, the proposed hybrid approach consistently outperforms the model-free approach in terms of the final cost achieved as well as the sample complexity. 

\begin{figure}
    \centering
    \begin{subfigure}[h]{0.5\textwidth}
    \centering
    \includegraphics[width = .9\textwidth]{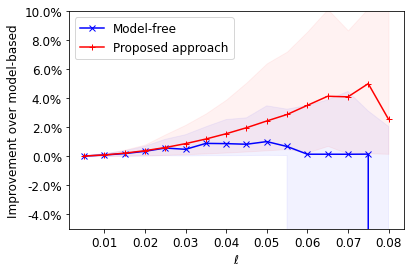}
    \caption{}\label{fig:toy_finalcost}
    \end{subfigure}~
        \begin{subfigure}[h]{0.5\textwidth}
        \centering
    \includegraphics[width = .9\textwidth]{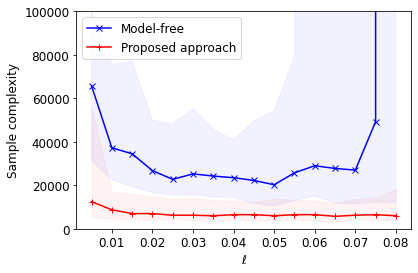}
    \caption{}\label{fig:toy_complexity}
    \end{subfigure}
\caption{Simulation results for synthetic experiments. Solid lines represent the median and shaded regions represent the $25\%$ to $75\%$ percentiles. }
    \label{fig:toy}
\end{figure}

\subsection{Inverted Pendulum}
\begin{wrapfigure}{R}{0.35\textwidth}
  \centering
  \includegraphics[width=\linewidth]{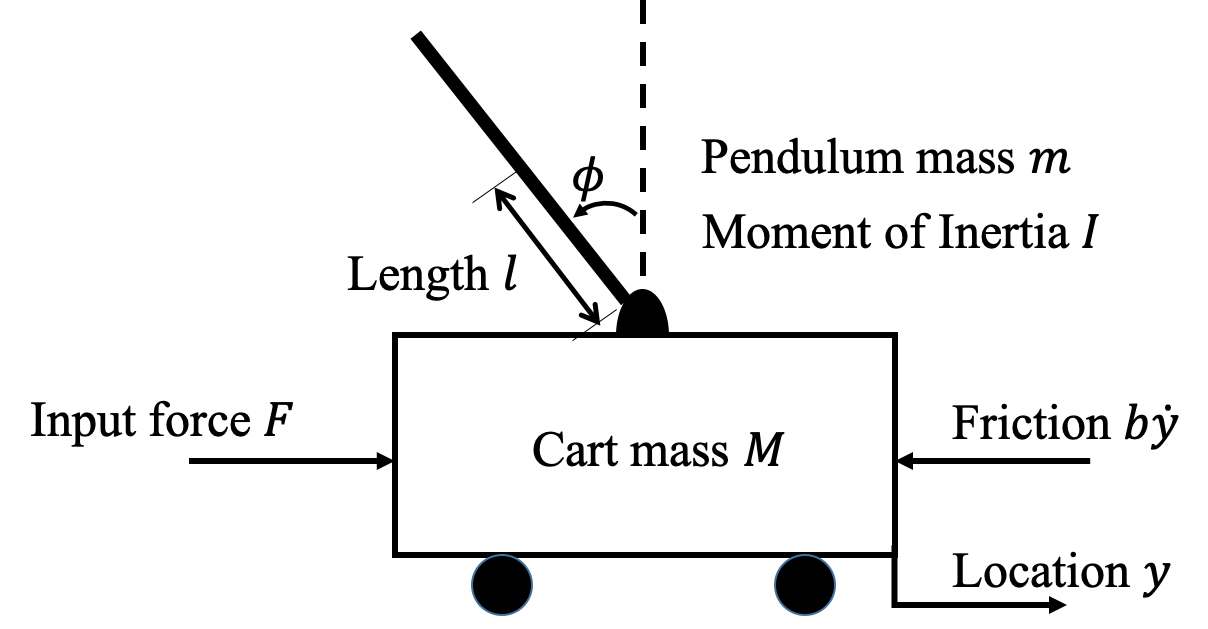}
  \caption{Cart inverted pendulum model with $M = \SI{0.5}{\kg}$, $m = \SI{0.2}{\kg}$, $b=\SI{0.1}{\newton \second\per\meter}$, $I =\SI{0.006}{\kg\meter\squared}$, $l = \SI{0.3}{\meter}$.}
  \label{fig:ip_model. }
\end{wrapfigure}

Our second set of experiments focuses on the cart inverted pendulum model (cf. Figure~\ref{fig:ip_model. }), where the goal is to stabilize the pendulum in the upright position. This is a nonlinear system with a widely accepted approximated linear model, and we provide its dynamics and its linearization below in continuous time \citep{magdy2019modeling}, 
{\small\begin{align*}
  \frac{d}{dt} \begin{bmatrix} y \\ \phi \\ \dot{y} \\ \dot{\phi} \end{bmatrix}
  & = \begin{bmatrix} \dot{y} \\ \dot{\phi} \\
        \begin{bmatrix} M + m & - m l \cos\phi \\ -m l \cos\phi &  I + m l^2 \end{bmatrix}^{-1} \begin{bmatrix} - b \dot{y} - m l (\dot{\phi})^2 \sin \phi + F \\ mgl \sin\phi \end{bmatrix} \end{bmatrix} \\
  & \approx \begin{bmatrix}
        0 & 0 & 1 & 0\\
        0 & 0 & 0 & 1 \\
        0 & \frac{m^2 g l^2}{I(M+m) + Mml^2} & \frac{-(I + ml^2)b}{I(M+m) + Mml^2} & 0 \\
        0 & \frac{mgl(M+m)}{I(M+m) + Mml^2} & \frac{-mlb}{I(M+m) + Mml^2} & 0
    \end{bmatrix} \begin{bmatrix} y \\ \phi \\ \dot{y} \\ \dot{\phi} \end{bmatrix} + \begin{bmatrix} 0 \\ 0 \\ \frac{I + ml^2}{I(M+m) + Mml^2} \\ \frac{ml}{I(M+m) + Mml^2} \end{bmatrix} F,
\end{align*}}where the ``$\approx$'' is obtained by setting $\sin\phi\approx \phi, \cos\phi \approx 1$ and $(\dot{\phi})^2\sin \phi \approx 0$. We identify the state as $x = [y,\phi,\dot{y},\dot{\phi}]^\top$ and the input as $u = F$. We discretize both the nonlinear system and the linear approximation above using forward discretization with the step size $\tau = 0.05s$ to obtain a discrete time nonlinear system and its approximation, and we set $f$ to be the difference of the two. We also set $Q = I, R = 1$, and the initial state distribution is a Dirac distribution centered on $x_0 = [0.8,0.8,0.2,0.2]^\top$. 

We run the proposed approach as well as the model-free approach, where the model-free approach is initialized at $K =  [k_1, k_2, k_3, k_3]$ which is generated randomly with $k_1, k_2$ drawn from $[-15, 0]$, and $k_3, k_4$ drawn from $[0,15]$.\footnote{Such an initialization is obtained through trial and error with the goal of ensuring a stabilizing initial controller with high probability. If the initial controller is unstable, we resample until it is stable. } For both approaches, we set the algorithm parameters as $\eta = 0.01$, $r = 0.001$, $T=2000$, $J=3$, $M=500$. We do 50 runs for both approaches, plot the learning processes in Figure~\ref{fig:ip_quantile}. We also plot the histogram of the final cost achieved by both approaches (normalized as the improvement over the model-based LQR controller) in Figure~\ref{fig:ip_hist}. 

The results show that the model-free approach fails to find a stabilizing controller in roughly 40\% of the runs, whereas almost all runs of the proposed approach can find a stabilizing controller,\footnote{We use a small number of trajectories for calculating the gradient ($J=3$).  As such, the proposed approach has a small probability of not converging when this gradient estimate is poor.} even though the model-free approach always starts from a stabilizing controller. Further, both the proposed hybrid approach and the model-free approach outperform the model-based LQR controller if they do reach a stabilizing controller. However, the proposed hybrid approach consistently achieves larger improvements than the model-free approach.

\begin{figure}
    \centering
    \begin{subfigure}[h]{0.5\textwidth}
    \centering
    \includegraphics[width = .9\textwidth]{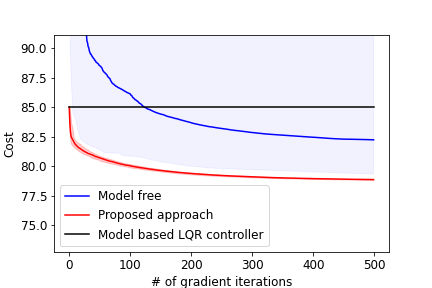}
    \caption{}\label{fig:ip_quantile}
    \end{subfigure}~
        \begin{subfigure}[h]{0.5\textwidth}
        \centering
    \includegraphics[width = .9\textwidth]{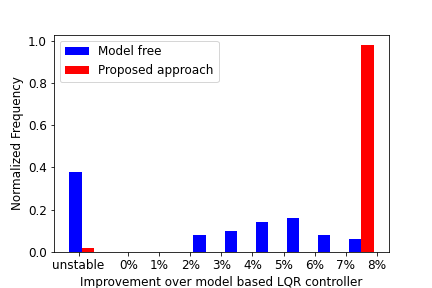}
    \caption{}\label{fig:ip_hist}
    \end{subfigure}
\caption{Simulation results for the inverted pendulum example. In (a), solid lines represent the median and  shaded regions represent the $10\%$ to $90\%$ percentiles. }
    \label{fig:ip}
\end{figure}

\newpage
\section*{Broader Impact}

This paper contributes to a growing literature that seeks to develop model-free approaches for learning that maintain provable convergence guarantees.  The algorithm proposed here provides guaranteed convergence, under specific assumptions, by merging ideas from model-free and model-based control.  This theoretical and algorithmic contribution to the literature has the potential to lead to improvements in a wide variety of non-linear control applications.  However, as is typical for theoretical contributions, the guarantees for the approach hold only under specific assumptions and so applications of the algorithm beyond those assumptions should proceed cautiously. 

We see no ethical concerns related to this paper.

{\small
\bibliographystyle{plainnat}
\bibliography{refs}}

\fvtest{\newpage
\appendix

\section{Explanation of Example~\ref{example:2}} \label{sec:explain}
The intuition behind \Cref{example:2} is as follows.
Let $A$ be contractive ($\norm{A} < 1$), but very close to $I$. As a result, with $f = 0$ and $K = 0$, any starting point $x$ will linearly converges to 0 (with a slow rate), thus incurring finite cost.
We construct a new contracting point close to $x_i$ by adding the following expression to function $f$:
\[\frac{\alpha_i (x_i - x) - (A - I) x - B K_i x}{(\norm{x - x_i}^2 + 1)^2}.\]
This contracting point has strength $\alpha_i$ and is effective when the policy $K$ is close to $K_i$. With such a function $f$, if we start from certain states, the state will converge to this contracting point $x_i$, thus incurring infinite cost.

In \Cref{example:2}, we construct three such contracting points, taking effects around different policy $K$, so that the activated policies (those incurring infinite cost) form a ring shape, leaving the center inactivated (incurring finite cost).
To visualize whether a policy is activated (incurring infinite cost), we compute the limit point to which the state converges under this policy, as shown in \Cref{fig:limit}.

\begin{figure}[h]
  \centering
  \includegraphics[width = 0.5 \textwidth]{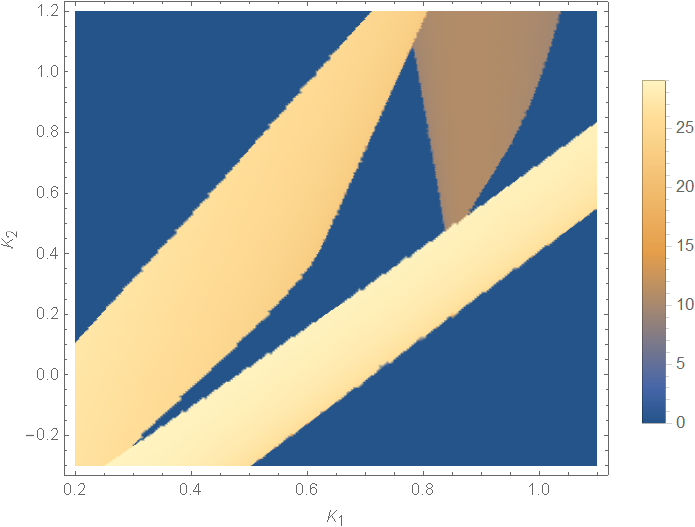}
  \caption{The squared norm of the limit state (we use $x_{100}$ for the plot) under different policies. For policies with infinite cost, the state converges to a non-zero contracting point. Thus, the points with non-zero value in this figure are exactly those with infinite value in \Cref{fig:disconnected}.}
  \label{fig:limit}
\end{figure}

\section{Proof of Theorem~\ref{thm:landscape}: Landscape Analysis of $C(K)$}\label{sec:landscape}

The proof will be divided into three steps, corresponding to part (a), (b) and (c) of the Theorem respectively. 

\textbf{Step 1:} We show in Lemma~\ref{lem:stability_0} that when $K\in\Omega$ and $\ell$ is bounded, then the system will be globally exponentially stable, or in other words the state trajectory will geometrically decay to the origin regardless of the initial state. This also implies boundedness of $C$ within $\Omega$. The proof of \Cref{lem:stability_0} is given in \Cref{subsec:stability}. 
\begin{lemma}\label{lem:stability_0}
When ${\ell\leq \frac{1-\sr}{4\cst}}$ and when $K\in \Omega$, $\forall x_0\in\R^n$, the system trajectory satisfies $\Vert x_t\Vert \leq c\rho^t \Vert x_0\Vert$, with $c = 2\sr$ and $\rho = \frac{1+\sr}{2}$. As a consequence, we have $C(K)$ is finite in $\Omega$.   
\end{lemma}

\textbf{Step 2:} We provide an explicit characterization of the cost function $C(K)$ and its gradient $\nabla C(K)$, and show the following \Cref{lem:str_cvx}, indicating the strong convexity and smoothness of the cost function. The proof of \Cref{lem:str_cvx} is provided in \Cref{subsec:str_cvx}.
\begin{lemma} \label{lem:str_cvx}When $\delta, \ell,\ell'$ satisfy,
\begin{align*}
        \delta \leq   \frac{\sigma_x\sigma (1-\rho)^4}{96 \abk^5 c^7 D_0^2},\qquad
    \ell \leq  \frac{\sigma_x\sigma (1-\rho)^5}{192\abk^4 c^9 D_0^2}, \qquad
    \ell'\leq \frac{\sigma_x\sigma(1-\rho)^5}{192\abk^4 c^{12} D_0^3}, 
\end{align*}
then $\Lambda(\delta) = \{K: \Vert K- \lopt\Vert_F \leq \delta\}\subset \Omega$, and for all, $K,K'\in\Lambda(\delta)$,  
\begin{align*}
     C(K') - C(K) &\geq \tr (K'-K)^\top \nabla C(K) + \frac{\mu}{2}\Vert K' - K\Vert_F^2,\\
     C(K') - C(K) &\leq \tr (K'-K)^\top \nabla C(K) + \frac{h}{2}\Vert K' - K\Vert_F^2,
\end{align*}
where $\mu = \sigma_x \sigma$, $h = 5\frac{\abk^4 c^4 D_0^2}{(1-\rho)^2}$. This implies $C(K)$ is $\mu$-strongly convex and $h$-smooth in the set $\Lambda(\delta)$. 
\end{lemma}

\textbf{Step 3:} We show that when $K$ is outside of the interior of $\Lambda(\frac{\delta}{3})$, $C(K)$ is larger than $C(\lopt)$. The proof of Lemma~\ref{lem:global_opt} is in \Cref{subsec:stationary}. 

\begin{lemma}\label{lem:global_opt}
Under the conditions of Lemma~\ref{lem:str_cvx} and if further, $\ell,\ell'$ satisfies,
{\[\ell \leq \delta \frac{ \sigma \sigma_x (1-\rho)^4}{96 \abk^4 c^8 D_0^2},\qquad \ell' \leq \delta \frac{ \sigma \sigma_x (1-\rho)^4}{96 \abk^4 c^{11} D_0^3} ,\]}
then for all $K \in\Omega/ \Lambda(\frac{\delta}{3})$, $C(K) >C(\lopt)$. 
\end{lemma}
The above lemma shows that $C(K)$'s minimum must be achieved in set $\Lambda(\frac{\delta}{3})$ which lies in the interior of $\Lambda(\delta)$. Since $C(K)$ is strongly convex in $\Lambda(\delta)$, $C(K)$'s minimum in $\Omega$ must be uniquely achieved at a point $K^*\in\Lambda(\frac{\delta}{3})$, which is also the unique stationary point of $C(K)$ within $\Lambda(\delta)$. 

Finally, we summarize the requirements for $\ell, \ell'$ and $\delta$ in the above three lemmas and provide a condition for $\ell, \ell'$ and an estimate of $\delta$ below which satisfies all the conditions in Lemma~\ref{lem:stability_0}, \ref{lem:str_cvx}, \ref{lem:global_opt}, 
\begin{align*}
    \ell \leq   \frac{ (\sigma \sigma_x)^2 (1-\rho)^8}{96^2 \abk^9 c^{15} D_0^4} , \qquad  \ell' \leq   \frac{ (\sigma \sigma_x)^2 (1-\rho)^8}{96^2 \abk^9 c^{18} D_0^5} ,\qquad \delta = \frac{\sigma_x\sigma (1-\rho)^4}{96 \abk^5 c^7 D_0^2}.
\end{align*}
With this, the proof of \Cref{thm:landscape} is concluded.
\subsection{Proof of Lemma~\ref{lem:stability_0}: Stability of the Trajectories}\label{subsec:stability}

We in fact show a more general result in the following lemma, of which part (a) leads to \Cref{lem:stability_0}. 

\begin{lemma}\label{lem:stability}
Assume $K\in\Omega$ and $\ell\leq \frac{1-\sr}{4\cst}$. Then we have the following holds.
\begin{itemize}
    \item[(a)] For any $x_0\in\R^n$, $\Vert x_t\Vert \leq c \rho^t\Vert x_0\Vert$, where $c=2\cst$ and $\rho = \frac{\sr+1}{2}$.  
    \item[(b)] Let $\{x_t\}$ and $\{x_t'\}$ be the state tracjectories starting from $x_0\in\R^n$ and $x_0'\in\R^n$ respectively. Then, $\Vert x_t - x_t'\Vert \leq c\rho^t \Vert x_0 - x_0'\Vert$. A direct consequence is that $\Vert \frac{\partial x_t}{\partial x_0}\Vert \leq c\rho^t$. 
    \item[(c)] Again let $\{x_t\}$ and $\{x_t'\}$ be the state tracjectories starting from $x_0\in\R^n$ and $x_0'\in\R^n$. Then $\Vert  \frac{\partial x_{t}}{\partial x_0} - \frac{\partial x_{t}'}{\partial x_0'} \Vert \leq \frac{\ell' c^3}{1-\rho} \rho^{t-1}\Vert x_0 - x_0'\Vert. $
\end{itemize}
\end{lemma}
\begin{proof}
To prove part (a), we recursively expand the system trajectory as follows, 
\begin{align*}
    x_{t+1} = (A-BK)x_t + f(x_t) = (A-BK)^{t+1} x_0 + \sum_{k=0}^t (A-BK)^{t-k}f(x_k).
\end{align*}
Taking the norm, and using $K\in\Omega$ and the Lipschitz property of $f$, we have,
\begin{align}
    \Vert x_{t+1}\Vert \leq \cst\sr^{t+1}\Vert x_0\Vert + \sum_{k=0}^t \cst\sr^{t-k} \ell \Vert x_k\Vert. \label{eq:stability:induction}
\end{align}

We use the following simple proposition on nonnegative scalar sequences satisfying inequalities of the form in \eqref{eq:stability:induction}.
\begin{proposition}\label{prop:seq_ub}
If nonnegative sequence $ a_t$ is such that 
$ a_{t+1} \leq \alpha_0 \lambda_1^{t+1} + \sum_{k=0}^t \alpha_1 \lambda_2^{t-k} a_k $
where $\lambda_0,\lambda_1\in(0,1)$ and $\alpha_0\geq a_0$
Then, $a_t\leq \alpha\lambda^t$ where $\alpha,\lambda$ can be any positive constant satisfying $\lambda>\lambda_2$, $\lambda\geq \lambda_1$, $\frac{\alpha_0}{\alpha} + \frac{\alpha_1}{\lambda-\lambda_2}\leq 1.$ In particular, we can pick $\alpha = 2\alpha_0$, and $\lambda=\max(\lambda_1,\lambda_2 + 2\alpha_1)$.
\end{proposition}
    \begin{proof}
    We use induction. The proposotion is clear true for $t=0$ as $\alpha\geq \alpha_0\geq a_0$. Assume it is true for $t$, then,
    \begin{align*}
        \frac{a_{t+1}}{\alpha\lambda^{t+1}} & \leq \frac{\alpha_0}{\alpha}(\frac{\lambda_1}{\lambda})^{t+1} + \sum_{k=0}^t \alpha_1 \lambda_2^{t-k}  \lambda^{k- t - 1}\\
        &= \frac{\alpha_0}{\alpha}(\frac{\lambda_1}{\lambda})^{t+1} +  \frac{ \alpha_1}{\lambda} \frac{1 - (\frac{\lambda_2}{\lambda})^{t+1} }{1 - \frac{\lambda_2}{\lambda}} < \frac{\alpha_0}{\alpha} + \frac{\alpha_1}{\lambda-\lambda_2}\leq 1.
    \end{align*}
    
    \end{proof}

Applying \Cref{prop:seq_ub} to \eqref{eq:stability:induction}, we have $\Vert x_t\Vert \leq  2c_0\Vert x_0\Vert (\sr+ 2 c_0\ell)^t \leq c\rho^t\Vert x_0\Vert $, where we have used $\sr+ 2 c_0\ell \leq \sr + 2\cst \frac{1-\sr}{4\cst} = \rho $.

The proof of part (b) is identical. Notice that 
\begin{align*}
  x_{t+1} - x_{t+1}'
  & = (A-BK)(x_t - x_t') + f(x_t) - f(x_t') \\
  & = (A-BK)^{t+1}(x_0 - x_0') + \sum_{k=0}^t (A-BK)^{t-k}( f(x_k) - f(x_k')).
\end{align*}
As such, 
\[\Vert x_{t+1} - x_{t+1}' \Vert \leq \cst\sr^{t+1}\Vert x_0 - x_0'\Vert + \sum_{k=0}^t \cst\sr^{t-k}\ell \Vert x_k - x_k'\Vert,\]
which leads to $\Vert x_t - x_t'\Vert \leq 2c_0\Vert x_0 - x_0'\Vert (\sr+ 2\cst\ell)^t \leq c\rho^t \Vert x_0 - x_0'\Vert$.

For part (c), we have
\begin{align*}
    \frac{\partial x_{t+1}}{\partial x_0} = (A-BK)\frac{\partial x_t}{\partial x_0} + \frac{\partial f(x_t)}{\partial x_t} \frac{\partial x_t}{\partial x_0},
\end{align*}
and therefore, 
\begin{align*}
    \frac{\partial x_{t+1}}{\partial x_0} - \frac{\partial x_{t+1}'}{\partial x_0'} & = (A-BK)(\frac{\partial x_t}{\partial x_0} -\frac{\partial x_{t}'}{\partial x_0'})+ \frac{\partial f(x_t)}{\partial x_t} (\frac{\partial x_t}{\partial x_0} - \frac{\partial x_{t}'}{\partial x_0'})+ ( \frac{\partial f(x_t)}{\partial x_t}  - \frac{\partial f(x_t')}{\partial x_t'}  )\frac{\partial x_{t}'}{\partial x_0'}\\
    &= \sum_{k=0}^t (A-BK)^{t-k} \Big[\frac{\partial f(x_k)}{\partial x_k} (\frac{\partial x_k}{\partial x_0} - \frac{\partial x_{k}'}{\partial x_0'}) + ( \frac{\partial f(x_k)}{\partial x_k}  - \frac{\partial f(x_k')}{\partial x_k'}  )\frac{\partial x_{k}'}{\partial x_0'}\Big]
\end{align*}
Taking the norm and using the Lipschitz continuity of $\frac{\partial f(x)}{\partial x}$ in Assumption~\ref{assump:ell}, we get
\begin{align*}
    \Vert  \frac{\partial x_{t+1}}{\partial x_0} - \frac{\partial x_{t+1}'}{\partial x_0'} \Vert&\leq \sum_{k=0}^t \cst\sr^{t-k} \Big[\ell \Vert \frac{\partial x_k}{\partial x_0} - \frac{\partial x_{k}'}{\partial x_0'} \Vert + \ell'\Vert x_k - x_k'\Vert \Vert\frac{\partial x_{k}'}{\partial x_0'} \Vert\Big]\\
    &\leq \sum_{k=0}^t \cst\sr^{t-k} \ell \Vert \frac{\partial x_k}{\partial x_0} - \frac{\partial x_{k}'}{\partial x_0'} \Vert + \sum_{k=0}^t \cst\sr^{t-k} \ell' (c\rho^k)^2 \Vert x_0 - x_0'\Vert \\
    &\leq \sum_{k=0}^t \cst\sr^{t-k} \ell \Vert \frac{\partial x_k}{\partial x_0} - \frac{\partial x_{k}'}{\partial x_0'} \Vert + \sum_{k=0}^t \cst  \ell' c^2 \rho^{t+k} \Vert x_0 - x_0'\Vert \\
    &< \sum_{k=0}^t \cst\sr^{t-k} \ell \Vert \frac{\partial x_k}{\partial x_0} - \frac{\partial x_{k}'}{\partial x_0'} \Vert +  \cst  \ell' c^2 \Vert x_0 - x_0'\Vert \frac{\rho^t}{1-\rho}. 
\end{align*}
With this, we can invoke \Cref{prop:seq_ub} and show that,
\[\Vert  \frac{\partial x_{t}}{\partial x_0} - \frac{\partial x_{t}'}{\partial x_0'} \Vert\leq  \frac{2 \cst\ell' c^2 }{\rho(1-\rho)}\rho^{t} \Vert x_0 - x_0'\Vert = \frac{\ell' c^3}{1-\rho} \rho^{t-1}\Vert x_0 - x_0'\Vert . \]
\end{proof} 

With the trajectory geometrically converging to zero, we also provide the following two auxiliary lemmas that will be used in the rest of the proof.  
\begin{lemma} \label{lem:sigma_ub}
For $K\in\Omega$, define
\[\Sigma_K = \E_{K} \sum_{t=0}^\infty x_t x_t^\top,\qquad \Sigma_K^{fx} = \E_{K} \sum_{t=0}^\infty f(x_t) x_t^\top.  \]
Then, under the same conditions as in Lemma~\ref{lem:stability_0}, we have,
\[\Vert \Sigma_K\Vert  \leq C_\Sigma: =  \frac{c^2 D_0^2}{1-\rho},\qquad \Vert \Sigma_K^{fx}\Vert  \leq \ell C_\Sigma.\]
\end{lemma}
\begin{proof}
As a direct consequence of Lemma~\ref{lem:stability_0},
\[ \Vert \Sigma_K\Vert \leq \E_{K} \sum_{t=0}^\infty \Vert x_t\Vert^2 \leq \frac{c^2}{1-\rho^2} \E \Vert x_0\Vert^2 \leq \frac{c^2 D_0^2}{1-\rho}. \]
Similarly, using the Lipschitz continuity of $f$,  
\[ \Vert \Sigma_K^{fx}\Vert \leq \E_{K} \sum_{t=0}^\infty \ell \Vert x_t\Vert^2 \leq \ell \frac{c^2 D_0^2}{1-\rho} .\]
\end{proof}
\begin{lemma} \label{lem:pk_ub} 
For $K\in\Omega$, let $P_K$ be the solution to the following Lyapunov equation, 
\[(A-BK)^\top P_K(A-BK) - P_K + Q + K^\top R K = 0. \]
Then, under the conditions of Lemma~\ref{lem:stability_0}, and further when $K\in \Lambda(1) $, we have \[\Vert P_K\Vert \leq C_P := \frac{c^2}{1-\rho} \Gamma^2 .\]
\end{lemma} 
\begin{proof} Note that 
$P_K= \sum_{t=0}^\infty ((A-BK)^\top)^t(Q + K^\top R K) (A-BK)^t$, we have
\begin{align*}
    \Vert P_K\Vert & \leq \frac{\cst^2}{1-\sr^2} \Vert Q + K^\top R K\Vert \leq \frac{\cst^2}{1-\sr}  (1 +   \Vert K\Vert^2 )\\
    & \leq \frac{\cst^2}{1-\sr} 5\abk^2 < \frac{c^2}{1-\rho} \Gamma^2 := C_P,
\end{align*}
where we have used $\Vert K\Vert\leq \Vert K - \lopt\Vert + \Vert\lopt\Vert \leq \Vert K - \lopt\Vert_F + \Vert\lopt\Vert \leq \Vert\lopt\Vert + 1 \leq  2\Gamma$. 
\end{proof}

\subsection{Proof of Lemma~\ref{lem:str_cvx}: Strong Convexity and Smoothness}\label{subsec:str_cvx}

First off, note that under the conditions of Lemma~\ref{lem:str_cvx}, the conditions in \Cref{lem:stability_0} are satisfied, and we can use all the results in \Cref{subsec:stability}, incluidng \Cref{lem:stability}, \Cref{lem:sigma_ub} and \Cref{lem:pk_ub}. Further, it is easy to check that the conditions in this lemma also guarantees $\Lambda(\delta) = \{K: \Vert K- \lopt\Vert_F \leq \delta\}\subset \Omega$ (which only requires $\delta \leq \frac{1-\sr}{\cst\abk}$).

In the following, we provide a characterization of the value function, the gradient, and provide a cost differential formula. Here the value and $Q$ function under a given controller $K$ are defined as,
\[ V_K(x) = \mathbb{E}_K\Big[ \sum_{t=0}^\infty x_t^\top Q x_t + u_t^\top R u_t\Big| x_0 = x\Big],\]
and 
\[ Q_K(x,u) = \mathbb{E}_K\Big[ \sum_{t=0}^\infty x_t^\top Q x_t + u_t^\top R u_t\Big| x_0 = x,u_0=u\Big] = x^\top Q x + u^\top R u + V_K(Ax + Bu + f(x)).\]
The following lemma provides a characterization of the value function. The proof of \Cref{lem:value_func} is given in \Cref{subsec:value_grad_ckk}.

\begin{lemma} [Value Function] \label{lem:value_func} When $K\in\Omega$, we have,
\begin{align}
    V_K(x) = x^\top  P_K x + g_K(x)
\end{align}
where $P_K$ is the solution to the following Lyapunov equation,
\begin{align}
 (A-BK)^\top P_K(A-BK) - P_K + Q + K^\top R K = 0,   \label{eq:lyapunov_eq}
\end{align}
and function $g_K$ is given by,
\begin{align}
    g_K(x) = 2 \tr P_K(A-BK) \sum_{t=0}^\infty x_t f(x_t)^\top + \tr P_K  \sum_{t=0}^\infty f(x_t) f(x_t)^\top,
\end{align}
where $\{x_t\}_{t=0}^\infty$ is the trajectory generated by controller $K$ with initial state $x_0 = x$. Further, when $K\in\Lambda(\delta)$, and when $ x,x'\in\mathbb{R}^n$ with $\Vert x\Vert, \Vert x'\Vert \leq 2c^2 D_0$, we have, \begin{align*}
    \Vert \nabla g_K(x) - \nabla g_K(x') \Vert& \leq L \Vert x - x'\Vert,
\end{align*}
where $L = (\ell + 2\ell'c^3 D_0) \frac{4C_P c^4}{(1-\rho)^2} = (\ell + 2\ell'c^3 D_0) \frac{4 \abk^2 c^6}{(1-\rho)^3}$ with $C_P$ being the upper bound on $\Vert P_K\Vert $ from Lemma~\ref{lem:pk_ub}. 
\end{lemma}

Given that $C(K) = \E_{x\sim \mathcal{D}} V_K(x)$, the formula for $V_K(x)$ in the preceding \Cref{lem:value_func} also leads to a formula for the gradient of $C(K)$, which is formally provided in the following lemma, whose proof is postponed to \Cref{subsec:value_grad_ckk}.

\begin{lemma}[Gradient of $C(K)$] \label{lem:grad}
Recall the cost function is 
$C(K) = \E_{x\sim \mathcal{D}} V_K(x)$. We have, 
\[\nabla C(K)  = 2E_K \Sigma_{K} - 2 B^\top P_K \Sigma_K^{fx} - B^\top \Sigma_K^{gx} \]
where $E_K$, $\Sigma_K$, $\Sigma_K^{fx}$ and $\Sigma_K^{gx}$ are defined as:
\begin{align}
    E_K &=RK - B^\top P_K(A-BK) = (R+ B^\top P_K B) K - B^\top P_K A ,\\
    \Sigma_K &= \E_{K} \sum_{t=0}^\infty x_tx_t^\top, \quad \Sigma_K^{fx} = \E_{K} \sum_{t=0}^\infty f(x_t) x_t^\top, \quad  \Sigma_K^{gx} = \E_{K} \sum_{t=0}^\infty \nabla_x g_K(x_{t+1}) x_t^\top.
\end{align}
\end{lemma}

We also provide a formula for $C(K') - C(K)$, whose proof can be found in \Cref{subsec:value_grad_ckk}. 
\begin{lemma}[Cost differential formula]\label{lem:ckk}
We have for any $K,K'\in \Omega$,
\begin{align}
    &C(K') - C(K) \nonumber \\
    & = 2\tr (K'-K)^\top E_K \Sigma_{K'} +  \tr (K'-K)^\top (R+B^\top P_KB)(K'-K) \Sigma_{K'}   - 2 \tr (K'-K)^\top B^\top P_K \Sigma_{K'}^{fx} \nonumber\\
    &\quad + \E_{K'} \sum_{t=0}^\infty \Big[ g_K((A-BK')x_t'+ f(x_t')) - g_K((A-BK)x_t'+f(x_t'))\Big].
\end{align}
\end{lemma}


With these preparations, we now proceed to prove Lemma~\ref{lem:str_cvx}, the strong convexity and smoothness of $C(K)$ within $\Lambda(\delta)$. 

\textit{Proof of Lemma~\ref{lem:str_cvx}:} We first focus on the strong convexity. By Lemma~\ref{lem:ckk}, we have for $K,K'\in\Lambda(\delta)$,
\begin{align}
    &C(K') - C(K) \nonumber \\
    &= 2\tr (K'-K)^\top E_K \Sigma_{K'} + \tr (K'-K)^\top [R + B^\top P_K B] (K'-K) \Sigma_{K'} -2 \tr(K'-K)^\top B^\top P_K \Sigma_{K'}^{fx}\nonumber\\
    &\quad +\E_{K'} \sum_{t=0}^\infty [g_K((A-BK')x_t' + f(x_t')) - g_K((A-BK)x_t' + f(x_t'))]\nonumber\\
    &\stackrel{(a)}{\geq} 2\tr (K'-K)^\top E_K \Sigma_{K} + 2\tr (K'-K)^\top E_K (\Sigma_{K'} - \Sigma_K) + \tr (K'-K)^\top [R + B^\top P_K B] (K'-K) \Sigma_{K'} \nonumber\\
    &\quad -2 \tr(K'-K)^\top B^\top P_K \Sigma_{K}^{fx} +2 \tr(K'-K)^\top B^\top P_K( \Sigma_{K}^{fx} - \Sigma_{K'}^{fx} ) \nonumber\\
    &\quad+ \E_{K'} \sum_{t=0}^\infty [ -\tr (K'-K)^\top B^\top \nabla g_K(x_{t+1}')x_t'^\top - \frac{L}{2} \Vert B(K'-K) x_t'\Vert^2  ]\nonumber\\
    &= \tr (K'-K)^\top \Big[2E_K\Sigma_K - 2 B^\top P_K\Sigma_K^{fx} - \E_{K} \sum_{t=0}^\infty B^\top \nabla g_K(x_{t+1}) x_t^\top \Big] \nonumber\\
    &\quad + \tr (K'-K)^\top [R + B^\top P_K B] (K'-K) \Sigma_{K'} \nonumber\\
    &\quad + 2\tr (K'-K)^\top E_K (\Sigma_{K'} - \Sigma_K)+ 2 \tr(K'-K)^\top B^\top P_K( \Sigma_{K}^{fx} - \Sigma_{K'}^{fx} )\nonumber\\
    &\quad + \tr(K'-K)^\top B^\top [\E_{K} \sum_{t=0}^\infty \nabla g_K(x_{t+1})x_t^\top - \E_{K'} \sum_{t=0}^\infty \nabla g_K(x_{t+1}')x_t'^\top  ] \nonumber\\
    &\quad - \E_{K'}\sum_{t=0}^\infty    \frac{L}{2} \Vert B(K'-K) x_t'\Vert^2  ]\nonumber\\
    &\stackrel{(b)}{\geq} \tr (K'-K)^\top \nabla C(K)   + \tr (K'-K)^\top [R + B^\top P_K B] (K'-K) \Sigma_{K'} \nonumber\\
    &\quad - 2 \Vert K'-K\Vert_F \Vert E_K\Vert \Vert\Sigma_{K'} - \Sigma_K\Vert_F - 2  \Vert K'-K \Vert_F \Vert B\Vert \Vert P_K\Vert \Vert \Sigma_{K}^{fx} - \Sigma_{K'}^{fx} \Vert_F \nonumber\\
    &\quad - \Vert K'-K\Vert_F \Vert B\Vert  \Big\Vert\E_{K} \sum_{t=0}^\infty \nabla g_K(x_{t+1})x_t^\top - \E_{K'} \sum_{t=0}^\infty \nabla g_K(x_{t+1}')x_t'^\top  \Big\Vert_F \nonumber\\
    &\quad - \E_{K'}\sum_{t=0}^\infty    \frac{L}{2} \Vert B(K'-K) x_t'\Vert^2  ] \label{eq:str_cvx_ckk_1} \end{align}
where in step (b) we have used the gradient formula in Lemma~\ref{lem:grad}, and in step (a) we have used, 
\begin{align*}
    &g_K((A-BK)x_t' + f(x_t')) \\
    &\leq g_K((A-BK')x_t' + f(x_t')) + \langle \nabla g_K((A-BK')x_t' + f(x_t')) , B(K'-K) x_t'  \rangle + \frac{L}{2} \Vert B(K'-K) x_t'\Vert^2\\
    &= g_K((A-BK')x_t' + f(x_t')) + \langle \nabla g_K(x_{t+1}'),B(K'-K) x_t' \rangle  + \frac{L}{2} \Vert B(K'-K) x_t'\Vert^2\\
    &= g_K((A-BK')x_t' + f(x_t')) + \tr    (K'-K)^\top B^\top  \nabla g_K(x_{t+1}') x_t'^\top   + \frac{L}{2} \Vert B(K'-K) x_t'\Vert^2.
\end{align*}
In the above, we have used the second part of Lemma~\ref{lem:value_func} on the Lipschitz continuity of $\nabla g_K(x)$, which applies here as $K\in\Lambda(\delta)$ and since $\Vert (A-BK')x_t' + f(x_t')\Vert =\Vert x_{t+1}'\Vert \leq c\rho^{t+1}\Vert x_0'\Vert \leq cD_0$, and $\Vert (A-BK)x_t' + f(x_t')\Vert \leq \Vert A-BK\Vert \Vert x_t'\Vert + \Vert f(x_t')\Vert \leq (c + \ell) c \rho^t \Vert x_0'\Vert \leq  2c^2 D_0$ (using  $\ell \leq 1 \leq c$). 
    
    Equation \eqref{eq:str_cvx_ckk_1} can lead to strong convexity if we can show its first two terms dominates its last 4 terms. For this purposes, we show the following Lemma \ref{lem:ek} and \ref{lem:ccc} to control the last 4 terms in \eqref{eq:str_cvx_ckk_1}. The proofs of Lemma \ref{lem:ek} and \ref{lem:ccc} can be found in Section~\ref{subsec:ek} and Section~\ref{subsec:ccc} respectively. 
\begin{lemma}  \label{lem:ek}
For $K\in\Lambda(\delta)$, we have, 
 \[\Vert E_K\Vert \leq C_E \Vert K - \lopt\Vert\leq C_E \Vert K - \lopt\Vert_F\leq C_E\delta, \]
 where $C_E = 4 \frac{\abk^4 c^4}{(1-\rho)^2} $. 
\end{lemma}

\begin{lemma}\label{lem:ccc} There exists constant $C_1 = \frac{2 c^3 \abk D_0^2}{(1-\rho)^2} , C_2  = \ell C_1, C_3 =L C_1$ such that for all $K,K'\in \Lambda(\delta)$, 
\begin{align*} 
    \Vert \Sigma_{K'} - \Sigma_K \Vert_F \leq C_1 \Vert K' - K\Vert_F, \quad \Vert \Sigma_{K'}^{fx} - \Sigma_K^{fx} \Vert_F \leq C_2 \Vert K' - K\Vert_F,
\end{align*}
\begin{align*}
    \Vert \E_{K} \sum_{t=0}^\infty \nabla g_K(x_{t+1})x_t^\top - \E_{K'} \sum_{t=0}^\infty \nabla g_K(x_{t+1}')x_t'^\top \Vert_F\leq C_3 \Vert K-K'\Vert_F.
\end{align*}
\end{lemma}
With the help of Lemma~\ref{lem:ek} and Lemma~\ref{lem:ccc}, we proceed with \eqref{eq:str_cvx_ckk_1},
    \begin{align}
    &C(K') - C(K) \nonumber \\
    &\geq \tr (K'-K)^\top \nabla C(K) + \tr (K'-K)^\top [R + B^\top P_K B] (K'-K) \Sigma_{K'} \nonumber \\
    & - 2 C_1 \Vert E_K\Vert \Vert K'-K\Vert_F^2 - 2C_2\Vert B\Vert \Vert P_K\Vert \Vert K'-K\Vert_F^2- C_3 \Vert B\Vert \Vert K'-K\Vert_F^2 - \E_{K'}\sum_{t=0}^\infty    \frac{L}{2} \Vert B\Vert^2\Vert x_t'\Vert^2  \Vert K'-K\Vert^2 \nonumber \\
    &\geq \tr (K'-K)^\top \nabla C(K) + \tr (K'-K)^\top [R + B^\top P_K B] (K'-K) \Sigma_{K'} \nonumber \\
    & \quad - \Bigg[ 2 C_1 \Vert E_K\Vert  + 2C_2\Vert B\Vert \Vert P_K\Vert   + C_3 \Vert B\Vert + \E_{K'}\sum_{t=0}^\infty    \frac{L}{2} \Vert B\Vert^2\Vert x_t'\Vert^2  \Bigg]\Vert K'-K\Vert_F^2 \nonumber \\
    &\geq \tr (K'-K)^\top \nabla C(K) + \mu\Vert K'-K\Vert_F^2  \nonumber\\
    &\quad  - \Big[ 2 C_1 C_E \delta  + 2C_2 \abk C_P   + C_3 \abk +   \frac{L}{2}  \frac{\abk^2 c^2 D_0^2}{1-\rho}  \Big]\Vert K'-K\Vert_F^2, \label{eq:str_cvx_ckk_2}
\end{align}
where in the last inequality, $\mu = \sigma_x \sigma$, and we have used since $P_K\succeq Q$ and $\Sigma_{K'} \succeq \E x_0 x_0^\top \succeq \sigma_x I$,
\begin{align*}
    \tr (K'-K)^\top [R + B^\top P_K B] (K'-K) \Sigma_{K'} & = \tr [(K'-K) \Sigma_{K'}^{1/2}]^\top [R + B^\top P_K B] (K'-K) \Sigma_{K'}^{1/2}\\
    &\geq  \tr [(K'-K) \Sigma_{K'}^{1/2}]^\top [R + B^\top Q B] (K'-K) \Sigma_{K'}^{1/2}\\
    &\geq  \sigma \tr [(K'-K) \Sigma_{K'}^{1/2}]^\top  (K'-K) \Sigma_{K'}^{1/2}\\
    &= \sigma \tr (K'-K) \Sigma_{K'} (K' - K)^\top\\
    &\geq \sigma \sigma_x \Vert K' - K\Vert_F^2. \label{eq:str_cvx_ckk_3}
\end{align*}
From \eqref{eq:str_cvx_ckk_2}, it is clear that if we can show,
\begin{align}
    2 C_1 C_E \delta  + 2C_2 \abk C_P   + C_3 \abk +   \frac{L}{2}  \frac{\abk^2 c^2 D_0^2}{1-\rho} \leq \frac{\mu}{2},
\end{align}
then the $\mu$-strong convexity property is proven. It remains to check our selection of $\delta, \ell, \ell'$ is such that \eqref{eq:str_cvx_ckk_3} is true. Plug in $C_2 = \ell C_1$ and $C_3 = L C_1$, we have,
\begin{align*}
    2 C_1 C_E \delta  + 2C_2 \abk C_P   + C_3 \abk +   \frac{L}{2}  \frac{\abk^2 c^2 D_0^2}{1-\rho} 
    & \leq 2 C_1 C_E \delta + 2\ell \abk C_1 C_P + 2L \abk  C_1\\
    &\leq  2 C_1 C_E \delta + 16 \Gamma C_1 [\ell + \ell' c^3 D_0] \frac{C_P c^4}{(1-\rho)^2}\\
    &= 16 \frac{\abk^5 c^7 D_0^2}{(1-\rho)^4} \delta + 32\frac{\abk^4 c^9 D_0^2}{(1-\rho)^5}\ell + 32\frac{\abk^4 c^{12} D_0^3}{(1-\rho)^5} \ell' \leq \frac{\mu}{2},
\end{align*}
where in the last step, we have used,
\begin{align*}
    \delta &\leq \frac{\sigma_x\sigma}{6} \frac{(1-\rho)^4}{16\abk^5 c^7 D_0^2} =   \frac{\sigma_x\sigma (1-\rho)^4}{96 \abk^5 c^7 D_0^2},\\
    \ell &\leq \frac{\sigma_x\sigma}{6}  \frac{(1-\rho)^5}{32\abk^4 c^9 D_0^2}= \frac{\sigma_x\sigma (1-\rho)^5}{192\abk^4 c^9 D_0^2},\\
    \ell'&\leq \frac{\sigma_x\sigma}{6}\frac{(1-\rho)^5}{32\abk^4 c^{12} D_0^3}=\frac{\sigma_x\sigma(1-\rho)^5}{192\abk^4 c^{12} D_0^3}. 
\end{align*}
This concludes the proof for the strong convexity. The proof for the smoothness property is similar. We follow similar steps as in $\eqref{eq:str_cvx_ckk_1}$ but reverse the direction of inequalities, getting, 
\begin{align}
    &C(K') - C(K) \nonumber \\
    &= 2\tr (K'-K)^\top E_K \Sigma_{K'} + \tr (K'-K)^\top [R + B^\top P_K B] (K'-K) \Sigma_{K'} -2 \tr(K'-K)^\top B^\top P_K \Sigma_{K'}^{fx}\nonumber\\
    &\quad +\E_{K'} \sum_{t=0}^\infty [g_K((A-BK')x_t' + f(x_t')) - g_K((A-BK)x_t' + f(x_t'))]\nonumber\\
    &\leq 2\tr (K'-K)^\top E_K \Sigma_{K} + 2\tr (K'-K)^\top E_K (\Sigma_{K'} - \Sigma_K) + \tr (K'-K)^\top [R + B^\top P_K B] (K'-K) \Sigma_{K'} \nonumber\\
    &\quad -2 \tr(K'-K)^\top B^\top P_K \Sigma_{K}^{fx} +2 \tr(K'-K)^\top B^\top P_K( \Sigma_{K}^{fx} - \Sigma_{K'}^{fx} ) \nonumber\\
    &\quad+ \E_{K'} \sum_{t=0}^\infty [ -\tr (K'-K)^\top B^\top \nabla g_K(x_{t+1}')x_t'^\top + \frac{L}{2} \Vert B(K'-K) x_t'\Vert^2  ]\nonumber\\
    &\leq \tr (K'-K)^\top \nabla C(K)   + \tr (K'-K)^\top [R + B^\top P_K B] (K'-K) \Sigma_{K'} \nonumber\\
    &\quad + 2 \Vert K'-K\Vert_F \Vert E_K\Vert \Vert\Sigma_{K'} - \Sigma_K\Vert_F + 2  \Vert K'-K \Vert_F \Vert B\Vert \Vert P_K\Vert \Vert \Sigma_{K}^{fx} - \Sigma_{K'}^{fx} \Vert_F \nonumber\\
    &\quad + \Vert K'-K\Vert_F \Vert B\Vert  \Big\Vert\E_{K} \sum_{t=0}^\infty \nabla g_K(x_{t+1})x_t^\top - \E_{K'} \sum_{t=0}^\infty \nabla g_K(x_{t+1}')x_t'^\top  \Big\Vert_F \nonumber\\
    &\quad + \E_{K'}\sum_{t=0}^\infty    \frac{L}{2} \Vert B(K'-K) x_t'\Vert^2  ] \nonumber \\
    &\leq \tr (K'-K)^\top \nabla C(K)   + \tr (K'-K)^\top [R + B^\top P_K B] (K'-K) \Sigma_{K'} + \frac{\mu}{2} \Vert K' -K\Vert_F^2\nonumber\\
    &\leq \tr (K'-K)^\top \nabla C(K)   +  \frac{1}{2}(\mu + 2\Vert R+B^\top P_K B\Vert \Vert \Sigma_{K'}\Vert) \Vert K' -K\Vert_F^2.
    \end{align}
Using the upper bound on $\Vert P_K\Vert$ and $\Vert\Sigma_{K'}\Vert$ in Lemma~\ref{lem:pk_ub} and Lemma~\ref{lem:sigma_ub} respectively, we get 
\begin{align*}
    \mu + 2\Vert R+B^\top P_K B\Vert \Vert \Sigma_{K'}\Vert \leq \mu + 2(1 + \abk^2 \frac{c^2\abk^2}{1-\rho}) \frac{c^2 D_0^2}{1-\rho} \leq 5\frac{\abk^4 c^4 D_0^2}{(1-\rho)^2} = h. 
\end{align*}
As such, the cost function $C(K)$ is $h$ smooth within $\Lambda(\delta)$. This concludes the proof of Lemma~\ref{lem:str_cvx}. 

\subsubsection{Proof of Lemma~\ref{lem:value_func}, \ref{lem:grad}, \ref{lem:ckk}: Characterization of $C(K)$ and its Gradient. } \label{subsec:value_grad_ckk}

\begin{proof}[Proof of Lemma~\ref{lem:value_func}]
Since $K\in\Omega$, by Lemma~\ref{lem:stability_0}, we have $V_K(x) \leq \Vert Q + K^\top R K\Vert \frac{c^2}{1-\rho^2} \Vert x\Vert^2 $. As such, $V_K(x)$ is finite and satisfies So $V_K(x) \rightarrow 0 $ as $x \rightarrow 0$. 

By Bellman equation, the value function also satisfies,
\begin{equation} \label{eq:v_recursive}
  V_K(x) = x^\top (Q + K^\top R K) x + V_K( (A-BK)x + f(x)).
\end{equation}
Define $g_K(x) = V_K(x) - x^\top P_K x$, we have 
\[x^\top P_K x + g_K(x) = x^\top (Q + K^\top R K) x  + ((A-BK)x + f(x))^\top P_K ((A-BK)x + f(x)) + g_K(x_1),\]
where $x_1 = (A-BK)x + f(x)$. Since $P_K$ satisfies \eqref{eq:lyapunov_eq}, we have, 
\begin{align*}
    g_K(x) & = 2 f(x)^\top P_K (A-BK) x + f(x)^\top P_K f(x) + g_K(x_1)\\
    &= 2 \tr(P_K(A-BK) x f(x)^\top) + \tr P_K f(x) f(x)^\top + g_K(x_1)\\
    &= 2 \tr P_K(A-BK)  \sum_{t=0}^\infty x_t f(x_t)^\top + \tr P_K  \sum_{t=0}^\infty f(x_t) f(x_t)^\top,
\end{align*}
where $\{x_t\}_{t=0}^\infty$ is the trajectory generated by controller $K$ starting from $x_0 =x $. In the last step in the above equation, we have used $g_K(x_t) \rightarrow 0 $ as $t\rightarrow\infty$, which is due to $g_K(x) \rightarrow 0 $ as $x\rightarrow 0$ and $\Vert x_t\Vert \leq c\rho^t \Vert x\Vert \rightarrow 0 $ as $t\rightarrow\infty$. 


Next, we show the second part of the Theorem.  We first compute the gradient of $g_K(x)$ as follows,
\begin{align}
 & [\nabla g_K(x)]^\top \nonumber
  \\
  & = 2 \sum_{t=0}^\infty \Big[f(x_t)^\top  P_K (A-BK)+ x_t^\top (A-BK)^\top P_K \frac{\partial f(x_t)}{\partial x_t} \Big] \frac{\partial x_t}{\partial x} + 2 \sum_{t=0}^\infty f(x_t)^\top P_K \frac{\partial f(x_t)}{\partial x_t}\frac{\partial x_t}{\partial x} \nonumber\\
  &  = 2 \sum_{t=0}^\infty \Big[f(x_t)^\top  P_K (A-BK)+ x_{t+1}^\top P_K \frac{\partial f(x_t)}{\partial x_t} \Big] \frac{\partial x_t}{\partial x} .
\end{align}

To show that $\nabla g_K(x)$ is Lipschitz in $x$ when $K\in\Lambda(\delta)$ and $\Vert x\Vert \leq 2c^2 D_0$, we have for $x,x'$ satsfying $\Vert x\Vert, \Vert x'\Vert \leq 2c^2 D_0$, 
\begin{align}
    &\Vert\nabla g_K(x) - \nabla g_K(x')\Vert \nonumber \\
    &\leq 2 \sum_{t=0}^\infty \Big\Vert [f(x_t) - f(x_t')]^\top P_K(A-BK) + x_{t+1}^\top P_K\frac{\partial f(x_t)}{\partial x_t} - x_{t+1}'^\top P_K\frac{\partial f(x_t')}{\partial x_t'} \Big\Vert \Vert \frac{\partial x_t}{\partial x}\Vert \nonumber \\
    &\quad + 2\sum_{t=0}^\infty\Big\Vert f(x_t')^\top P_K (A-BK) + x_{t+1}'^\top  P_K\frac{\partial f(x_t')}{\partial x_t'} \Big\Vert \Vert \frac{\partial x_t'}{\partial x'} -\frac{\partial x_t}{\partial x}\Vert. \label{eq:g_lip_1}
    \end{align}
Using $\Vert \frac{\partial{f}(x)}{\partial x} - \frac{\partial{f}(x')}{\partial x'} \Vert \leq \ell' \Vert x - x'\Vert $ (Assumption~\ref{assump:ell}) and the fact that for any $t$, by Lemma~\ref{lem:stability_0}, $\Vert x_t'\Vert \leq c \Vert x'\Vert \leq 2 c^3 D_0:=D$, we have, 
\begin{align}
  &  \Vert x_{t+1}^\top P_K\frac{\partial f(x_t)}{\partial x_t} - x_{t+1}'^\top   P_K\frac{\partial f(x_t')}{\partial x_t'} \Vert \nonumber \\
    &\leq \Vert (x_{t+1} - x_{t+1}')^\top P_K\frac{\partial f(x_t)}{\partial x_t} \Vert +\Vert x_{t+1}'^\top P_K(\frac{\partial f(x_t')}{\partial x_t'} - \frac{\partial f(x_t)}{\partial x_t})\Vert \nonumber\\
    &\leq \ell C_P \Vert x_{t+1} - x_{t+1}'\Vert + D C_P \ell' \Vert x_t - x_t'\Vert \nonumber \\
    &\leq C_P(\ell + D \ell') c \Vert x - x'\Vert, \label{eq:g_lip_2}
\end{align}
where in the second last inequality, we have used the bound on $\Vert P_K\Vert$ when $K\in\Lambda(\delta)$ (cf. Lemma~\ref{lem:pk_ub}), and in the last inequality, we have used Lemma~\ref{lem:stability} (b). Further, we have,
\begin{align}
    \Big\Vert (f(x_t)- f(x_t'))^\top P_K (A-BK)  \Big\Vert \leq \ell \Vert x_t - x_t'\Vert C_P \Vert A-BK\Vert\leq \ell  C_P c^2 \Vert x - x'\Vert, \label{eq:g_lip_3}
\end{align}
where we have used $\Vert A-BK\Vert\leq \cst\leq c$.  Also notice,
\begin{align}
    \Big\Vert f(x_t')^\top P_K (A-BK) + x_{t+1}'^\top  P_K\frac{\partial f(x_t')}{\partial x_t'} \Big\Vert \leq \ell \Vert x_t'\Vert C_P c + \Vert x_{t+1}'\Vert C_P \ell \leq 2 \ell D C_P c. \label{eq:g_lip_4}
\end{align}
Plugging in \eqref{eq:g_lip_2}, \eqref{eq:g_lip_3}, \eqref{eq:g_lip_4} into \eqref{eq:g_lip_1}, and using $ \Vert \frac{\partial x_t}{\partial x}\Vert \leq c\rho^t$ (Lemma~\ref{lem:stability} (b)), $\Vert \frac{\partial x_t'}{\partial x'} -\frac{\partial x_t}{\partial x}\Vert \leq \frac{\ell' c^3}{(1-\rho)}\rho^{t-1} \Vert x - x'\Vert $ (Lemma~\ref{lem:stability} (c)),  we get,
\begin{align*}
    &\Vert\nabla g_K(x) - \nabla g_K(x')\Vert \\
    & \leq  2 \sum_{t=0}^\infty \Big[\ell  C_P c^2 \Vert x - x'\Vert+ C_P(\ell + D \ell') c \Vert x - x'\Vert \Big] \Vert \frac{\partial x_t}{\partial x}\Vert + 2 \sum_{t=1}^\infty 2 \ell D C_P c \Vert \frac{\partial x_t'}{\partial x'} -\frac{\partial x_t}{\partial x}\Vert\\
    &\leq  2   \Big[\ell  C_P c^2 \Vert x - x'\Vert+ C_P(\ell + D \ell') c \Vert x - x'\Vert \Big] \frac{c}{1-\rho} + 4\ell D C_P c  \frac{\ell' c^3}{(1-\rho)^2} \Vert x - x'\Vert \\
    &\leq \Big[ (2\ell + \ell' D) \frac{2 C_P c^3}{1-\rho} + 4 \ell \ell'D \frac{C_P c^4}{(1-\rho)^2} \Big]\Vert x - x'\Vert\\
    &\leq (\ell + \ell'D) \frac{4C_P c^4}{(1-\rho)^2} \Vert x - x'\Vert,
\end{align*}
where in the last inequality, we have used  $4\ell \leq 2$. This shows $\nabla g_K(x) $ is $L$-Lipschitz continuous in $x$. 
\end{proof}

\begin{proof}[Proof of Lemma~\ref{lem:grad}]
In \eqref{eq:v_recursive}, we take derivative of $V_K(x)$ w.r.t. $K$, and have
\[\nabla_K V_K(x) = 2RK x x^\top + \nabla_K V(x_1) + (\frac{\partial x_1}{\partial K}  )^\top \nabla_x V_K(x_1).\]
To proceed, note the directional derivative of $x_1$ w.r.t. $K$ in the direction of $\Delta$ is
$x_1'[\Delta] = -B\Delta x$. Therefore,
\[(x_1'[\Delta])^\top \nabla_x V_K(x_1) = - x^\top \Delta^\top B^\top [ 2P_K x_1 + \nabla_x g_K(x_1)] = \tr \Delta^\top (-2B^\top P_K x_1 x^\top  -  B^\top \nabla g_K(x_1) x^\top  )\]
This implies that
\begin{align*}
    \nabla_K V_K(x) &= 2RK x x^\top - 2B^\top P_K [ (A-BK)x+f(x)]x^\top - B^\top \nabla_x g_K(x_1) x^\top + \nabla_K V(x_1) \\
    &= (2RK -2B^\top P_K(A-BK)) xx^\top - 2B^\top P_K f(x)x^\top  - B^\top \nabla_x g_K(x_1)x^\top + \nabla_K V(x_1)\\
    &=2E_K \sum_{t=0}^\infty x_t x_t^\top - 2B^\top P_K\sum_{t=0}^\infty f(x_t)x_t^\top - B^\top \sum_{t=0}^\infty \nabla_x g_K(x_{t+1})x_t^\top ,
\end{align*}
where $\{x_t\}$ is the trajectory starting from $x_0=x$. Taking expectation w.r.t. $x_0$ and we are done. 
\end{proof}

\begin{proof}[Proof of Lemma~\ref{lem:ckk}]
By \cite[Lemma 10]{fazel2018global}, we have
\[V_{K'}(x) - V_{K}(x) = \sum_{t=0}^\infty A_K(x_t',u_t') \]
where $\{x_t',u_t'\}$ is the trajectory generated by $x_0'=x$ and $u_t' = -K' x_t'$, and $A_K(x,u) = Q_K(x,u) - V_K(x)$ is the advantage function \citep{kakade2003sample}. 

Now, for given $u= -K'x$, we have
\begin{align}
    & A_K(x,u) = Q_K(x,u) - V_K(x) \nonumber \\
    &=x^\top (Q + (K')^\top RK')x + V_K((A-BK')x + f(x)) - V_K(x) \nonumber\\
    &= x^\top (Q+(K-K+K')^\top R (K-K+K'))x  + V_K((A-BK')x + f(x)) - V_K(x)\nonumber\\
    &= x^\top (Q+K^\top R K)x + x^\top (2 (K'-K)^\top R K + (K'-K)^\top R(K'-K))x + V_K((A-BK')x + f(x)) - V_K(x) \nonumber\\
    &= x^\top (2 (K'-K)^\top R K + (K'-K)^\top R(K'-K))x + V_K((A-BK')x + f(x)) - V_K((A-BK)x + f(x)). \label{eq:ckk:advantage}
\end{align}
We next compute, using the formula for value function in \Cref{lem:value_func},
\begin{align*}
    &V_K((A-BK')x + f(x)) - V_K((A-BK)x + f(x))\\
    &= ((A-BK')x + f(x))^\top P_K((A-BK')x + f(x)) - ((A-BK)x + f(x))^\top P_K((A-BK)x + f(x)) \\
    &\quad + g_K((A-BK')x + f(x)) - g_K((A-BK)x + f(x))\\
    &= 2 (B(K-K')x)^\top P_K ((A-BK)x + f(x) ) + x^\top  (K-K')^\top B^\top  P_K B (K-K') x\\
     &\quad + g_K((A-BK')x + f(x)) - g_K((A-BK)x + f(x))\\ 
     &= 2 x^\top(K-K')^\top B^\top  P_K(A-BK) x + 2 x^\top (K-K')^\top B^\top P_K f(x) +x^\top  (K-K')^\top B^\top  P_K B (K-K') x\\
     &\quad + g_K((A-BK')x + f(x)) - g_K((A-BK)x + f(x)).
\end{align*}
Plugging the above into \eqref{eq:ckk:advantage}, we have
\begin{align*}
     A_K(x,u)&= 2\tr (K'-K)^\top [RK - B^\top  P_K(A-BK) ] x x^\top + \tr ((K'-K)^\top [ R + B^\top P_K B](K'-K) ) xx^\top \\
     &\quad - 2 \tr  (K'-K)^\top B^\top P_K f(x)x^\top + g_K((A-BK')x + f(x)) - g_K((A-BK)x + f(x)).
\end{align*}
As a result, we have,
\begin{align*}
    & C(K') - C(K) \\
    &= \E_{K'} \sum_{t=0}^\infty A_K(x_t',-K'x_t') \\
    &= 2\tr (K'-K)^\top E_K \Sigma_{K'} + \tr (K'-K)^\top [R + B^\top P_K B] (K'-K) \Sigma_{K'} -2 \tr(K'-K)^\top B^\top P_K \Sigma_{K'}^{fx}\\
    &\quad +\E_{K'} \sum_{t=0}^\infty [g_K((A-BK')x_t' + f(x_t')) - g_K((A-BK)x_t' + f(x_t'))].
\end{align*}
\end{proof}





\subsubsection{Proof of Lemma \ref{lem:ek}: bounds on $\Vert E_K\Vert$} \label{subsec:ek}
Note that $K\in\Lambda(\delta)$, $E_K = RK - B^\top P_K(A-BK)$. Further, by \citet{fazel2018global}, $E_{\lopt} = 0$. Then, we have ,
\begin{align*}
    &\Vert E_K \Vert   = \Vert E_K - E_{\lopt}\Vert \\
    &\leq \Vert R ( K - \lopt)\Vert + \Vert B^\top P_{\lopt} B(K-\lopt)\Vert + \Vert B^\top (P_K - P_{\lopt})(A-BK)\Vert\\
    &\leq (1 + \abk^2 C_P)\Vert K - \lopt\Vert + \abk c  \frac{2\Gamma^3 c^3}{(1-\rho)^2} \Vert K - \lopt\Vert\\
    &\leq 4 \frac{\abk^4 c^4}{(1-\rho)^2} \Vert K - \lopt\Vert,
\end{align*}
where in the second inequality, we have used Lemma~\ref{lem:P_perturbation} which is provided below. This concludes the proof of \Cref{lem:ek}

\begin{lemma}[Perturbation of $P_K$]\label{lem:P_perturbation} When $K\in \Lambda(\delta) $, we have, 
\begin{align*}
  \Vert P_K - P_{\lopt}\Vert \leq    \frac{2\Gamma^3 c^3}{(1-\rho)^2} \Vert K - \lopt\Vert
\end{align*}
\end{lemma}
\begin{proof}
Recall that $P_K = \sum_{t=0}^\infty ((A-BK)^{\top})^t (Q + K^\top R K) (A-BK)^t$. We calculate the direction derivative of $P_K$ w.r.t. $K$ in the direction of $\Delta$ when $K\in\Lambda(\delta)$,
\begin{align*}
    & P_K'[\Delta] \\
    &= \sum_{t=0}^\infty (((A-BK)^t)'[\Delta])^\top (Q+K^\top RK)(A-BK)^t + \sum_{t=0}^\infty ((A-BK)^t)^\top (Q+K^\top RK)((A-BK)^t)'[\Delta]\\
    &\quad + \sum_{t=0}^\infty ((A-BK)^t)^\top (\Delta^\top RK+K^\top R\Delta) (A-BK)^t.
\end{align*}

Notice that
\[((A-BK)^t)'[\Delta] = \sum_{k=1}^t (A-BK)^{k-1}(-B\Delta)(A-BK)^{t-k}.\]
Hence
\[\Vert ((A-BK)^t)'[\Delta] \Vert\leq\Vert B\Vert \Vert\Delta\Vert \cst^2 t \sr^{t-1}\leq\Vert B\Vert \cst^2\frac{2}{1-\sr} (\frac{1+\sr}{2})^{t}\Vert\Delta\Vert,\]
where we have used the fact that $t\sr^{t-1}\leq \frac{2}{1-\sr}(\frac{1+\sr}{2})^{t}$. As such, we have
\begin{align*}
    \Vert P_K'[\Delta]\Vert &\leq 2\sum_{t=0}^\infty \Vert B\Vert \Vert Q+ K^\top RK\Vert \cst\sr^t  \cst^2\frac{2}{1-\sr} (\frac{1+\sr}{2})^{t}\Vert\Delta\Vert +2 \sum_{t=0}^\infty \cst^2\sr^{2t} \Vert K^\top R\Vert \Vert\Delta\Vert\\
    &< 8 \Vert B\Vert \Vert Q+K^\top  RK\Vert \frac{\cst^3}{(1-\sr)^2}\Vert\Delta\Vert + 2 \Vert K^\top R\Vert \frac{\cst^2}{(1-\sr)} \Vert\Delta\Vert\\
    &\leq  (8 \abk   + 8 \abk \Vert K\Vert^2 + 2   \Vert K\Vert) \frac{\cst^3}{(1-\sr)^2} \Vert\Delta\Vert.
\end{align*}
We further use that $\Vert K\Vert \leq \Vert\lopt\Vert + \delta\leq 2\Gamma$ (using $\delta \leq 1\leq \Gamma $), then, we have, 
\begin{align*}
    \Vert P_K'[\Delta]\Vert \leq 44\Gamma^3  \frac{\cst^3}{(1-\sr)^2} \Vert\Delta\Vert \leq \frac{2\Gamma^3 c^3}{(1-\rho)^2} \Vert \Delta\Vert .
\end{align*}
Using a simple integration argument on the line between $\lopt$ and $K$, we have
\[ \Vert P_K - P_{\lopt} \Vert \leq \frac{2\Gamma^3 c^3}{(1-\rho)^2} \Vert K - \lopt\Vert.\]
\end{proof}

\subsubsection{Proof of Lemma~\ref{lem:ccc}: Bounds on $C_1, C_2, C_3$}\label{subsec:ccc}
Before we start the proof, we first provide an auxiliary result on the perturbation of trajectories by a change of controller $K$.
\begin{lemma}\label{lem:x_t_K_perturb}
For $K\in\Lambda(\delta)$, given $x_0$, the directional derivative of $x_t$ w.r.t. $K$ in the direction of $\Delta$ satisfies, 
\[\Vert x_t'[\Delta]\Vert \leq  \frac{c^2\abk}{1-\rho} \rho^t \Vert x_0\Vert \Vert \Delta\Vert. \]
As a direct consequence, for $K,K'\in\Lambda(\delta)$, let $\{x_t\}_{t=0}^\infty$ and $\{x_t'\}_{t=0}^\infty$ be two trajectories starting from the same $x_0 = x_0'$ generated by $K$ and $K'$ respectively. Then, we have $\Vert x_t - x_t'\Vert \leq \frac{c^2\Gamma}{1-\rho} \rho^t \Vert x_0\Vert   \Vert K'-K\Vert.$ 
\end{lemma}
\begin{proof}
The dynamical system is given by
\[x_{t+1}=(A-BK) x_t + f(x_t).\]
Taking derivative w.r.t. $K$ in the direction of $\Delta$, we have
\begin{align*}
    x_{t+1}'[\Delta] = (A-BK) x_t'[\Delta] - B\Delta x_t + \frac{\partial f(x_t)}{\partial x_t} x_t'[\Delta] = \sum_{k=0}^t (A-BK)^{t-k}[-B\Delta x_k +\frac{\partial f(x_k)}{\partial x_k} x_k'[\Delta] ].
\end{align*}
Taking the norm and using the triangle inequality as well as the Lipschitz property of $f$, we get,
\begin{align*}
    \Vert x_{t+1}'[\Delta]\Vert &\leq \sum_{k=0}^t \cst\sr^{t-k}\ell \Vert x'_k[\Delta]\Vert + \sum_{k=0}^t \cst\sr^{t-k}\Vert B\Delta\Vert \Vert x_k\Vert\\
    &\leq \sum_{k=0}^t \cst\sr^{t-k}\ell \Vert x'_k[\Delta]\Vert + \sum_{k=0}^t \cst\sr^{t-k}\Vert B\Delta\Vert c\rho^k\Vert x_0\Vert\\
    &\leq \sum_{k=0}^t \cst\sr^{t-k}\ell \Vert x'_k[\Delta]\Vert +  \cst\Vert B\Vert \Vert \Delta\Vert c\Vert x_0\Vert \frac{\rho^{t+1} - \sr^{t+1}}{\rho - \sr}.
\end{align*}
As such, by a simple induction argument (Proposition~\ref{prop:seq_ub}), we have
\[\Vert x'_t[\Delta]\Vert \leq 2 \cst\Vert B\Vert \Vert \Delta\Vert c\Vert x_0\Vert \frac{\rho^{t}}{\rho - \sr} =   \frac{c^2\abk}{1-\rho} \rho^t \Vert x_0\Vert \Vert \Delta\Vert.\]
\end{proof}

We now proceed to prove Lemma~\ref{lem:ccc}. 

\begin{proof}[Proof of Lemma~\ref{lem:ccc}]
By definition, $\Sigma_K = \sum_{t=0}^\infty  \E x_t x_t^\top $. For $K\in\Lambda(\delta)$, we take the directional derivative w.r.t. $K$ in the direction of $\Delta$, getting,
we have
\[\Sigma_K'[\Delta] =\E \sum_{t=0}^\infty   \Big(x_t'[\Delta] x_t^\top + x_t x_t'[\Delta]^\top\Big) .    \]
Then, using Lemma~\ref{lem:x_t_K_perturb}, we have,
\begin{align*}
    \Vert \Sigma_K'[\Delta] \Vert_F &\leq \E \sum_{t=0}^\infty 2\Vert x_t'[\Delta]\Vert \Vert x_t\Vert \\
    &\leq \E \sum_{t=0}^\infty 2 \frac{c^2\abk}{1-\rho} \rho^t \Vert x_0\Vert \Vert \Delta\Vert  c\rho^t \Vert x_0\Vert\\
    &\leq \frac{2 c^3 \abk D_0^2}{(1-\rho)^2}  \Vert \Delta\Vert\\
    &\leq \frac{2 c^3 \abk D_0^2}{(1-\rho)^2}  \Vert \Delta\Vert_F,
\end{align*}
which, after a simple integration argument, gives a bound for $C_1$. Next, we consider the bound on $C_2$. Note that
\[\Sigma_{K}^{fx} = \E \sum_{t=0}^\infty f(x_t) x_t^\top .\]
Again, taking the derivative, we have,
\[(\Sigma_{K}^{fx})'[\Delta] = \E \sum_{t=0}^\infty [\frac{\partial f(x_t)}{\partial x_t} x_t'[\Delta] x_t^\top + f(x_t) x_t'[\Delta]^\top    ],\]
which leads to,
\begin{align*}
    \Vert (\Sigma_{K}^{fx})'[\Delta] \Vert_F \leq \E \sum_{t=0}^\infty 2  \ell \Vert x_t\Vert \Vert  x_t'[\Delta]\Vert \leq  \ell C_1 \Vert \Delta\Vert_F.
\end{align*}
So we will get $C_2 = \ell C_1$.

Finally, we proceed to bound $C_3$. Recall the definition of $C_3$ is such that for $K,K'\in\Lambda(\delta)$,
\begin{align*}
    \Vert \E_{K} \sum_{t=0}^\infty \nabla g_K(x_{t+1})x_t^\top - \E_{K'} \sum_{t=0}^\infty \nabla g_K(x_{t+1}')x_t'^\top \Vert_F\leq C_3 \Vert K-K'\Vert_F.
\end{align*}
Fix $x_0$ for now with $\Vert x_0\Vert\leq D_0$, and consider the trajectories $\{x_t\}_{t=0}^\infty$ and $\{x_t'\}_{t=0}^\infty$ generated by controller $K$ and $K'$ starting from $x_0' = x_0$. We have,
\begin{align*}
    &\Vert \nabla g_K(x_{t+1})x_t^\top -  \nabla g_K(x_{t+1}')x_t'^\top\Vert_F  \\
    &\leq \Vert ( \nabla g_K(x_{t+1}) - \nabla g_K(x_{t+1}')) x_t^\top\Vert_F + \Vert \nabla g_K(x_{t+1}')(x_t - x_t')^\top\Vert_F\\
    &{\leq} \Vert  \nabla g_K(x_{t+1}) - \nabla g_K(x_{t+1}')\Vert \Vert x_t\Vert + \Vert \nabla g_K(x_{t+1}')\Vert\Vert x_t - x_t'\Vert\\
    &\stackrel{(a)}{\leq} L \Vert x_{t+1} - x_{t+1}'\Vert \Vert x_t\Vert + L \Vert x_{t+1}'\Vert \Vert x_t - x_t'\Vert \\
    &\stackrel{(b)}{\leq} L \frac{c^2\Gamma}{1-\rho} \rho^{t+1} c\rho^t \Vert  x_0\Vert^2 \Vert K'-K\Vert  + L c\rho^{t+1}  \frac{c^2\Gamma}{1-\rho} \rho^t    \Vert  x_0\Vert^2 \Vert K'-K\Vert\\
    &\leq L \frac{2c^3\abk D_0^2}{1-\rho}\rho^t \Vert K'-K\Vert, 
\end{align*}
where in inequality (a), we have used the Lipschitz continuity of $\nabla g_K(x)$ (Lemma~\ref{lem:value_func}), which holds here as $K\in\Lambda(\delta)$ and $\Vert x_{t+1}\Vert\leq c D_0, \Vert x_{t+1}'\Vert\leq cD_0$. In inequality (b), we have used the bound in Lemma \ref{lem:x_t_K_perturb}. With the above bound, we can proceed to obtain $C_3$, getting,
\begin{align*}
    &\Vert \E_{K} \sum_{t=0}^\infty \nabla g_K(x_{t+1})x_t^\top - \E_{K'} \sum_{t=0}^\infty \nabla g_K(x_{t+1}')x_t'^\top \Vert_F\\
    &\leq \sum_{t=0}^\infty \E_{K,K'} \Vert\nabla g_K(x_{t+1})x_t^\top  - \nabla g_K(x_{t+1}')x_t'^\top  \Vert_F\\
    &\leq   L \frac{2c^3\abk D_0^2}{(1-\rho)^2}   \Vert K'-K\Vert\\
    &\leq L C_1 \Vert K'-K\Vert_F.
\end{align*}
As a result, we can set $C_3 = LC_1 $.
\end{proof}

\subsection{Proof of Lemma~\ref{lem:global_opt}: Global Optimality} \label{subsec:stationary}

By Lemma~\ref{lem:ckk}, we have
\begin{align*}
    C(K') - C(K) & = 2\tr (K'-K)^\top E_K \Sigma_{K'} \nonumber \\
    & \quad + \tr (K'-K)^\top (R+B^\top P_KB)(K'-K) \Sigma_{K'} - 2 \tr (K'-K)^\top B^\top P_K \Sigma_{K'}^{fx} \nonumber \\
    & \quad + \E_{K'} \sum_{t=0}^\infty \Big[ g_K((A-BK')x_t'+ f(x_t')) - g_K((A-BK)x_t'+f(x_t'))\Big].
\end{align*}
Setting $K = \lopt$ in the above equation and using $E_{\lopt} = 0$ (cf. \citet{fazel2018global}), we get $\forall K\in\Omega$,
\begin{align}
    C(K) - C(\lopt) & =   \tr (K-\lopt)^\top (R+B^\top P_{\lopt}B)(K-\lopt) \Sigma_{K}    - 2 \tr (K-\lopt)^\top B^\top P_{\lopt} \Sigma_{K}^{fx} \nonumber \\
    &\quad + \E_{K} \sum_{t=0}^\infty \Big[ g_{\lopt}((A-BK)x_t+ f(x_t)) - g_{\lopt}((A-B\lopt)x_t+f(x_t))\Big] \nonumber \\
&\geq  \mu \Vert K - \lopt\Vert_F^2 
    - 2 \Vert K-\lopt\Vert_F \Vert B \Vert \Vert P_{\lopt}\Vert \Vert \Sigma_K^{fx}\Vert_F\nonumber\\
    &\quad + \E_{K} \sum_{t=0}^\infty \Big[ g_{\lopt}((A-BK)x_t+ f(x_t)) - g_{\lopt}((A-B\lopt)x_t+f(x_t))\Big], \label{eq:global_opt_ineq_1}
\end{align}
where in the last inequality, we have used that by $R+B^\top P_{\lopt} B \succeq \sigma I$, $\Sigma_K \succeq \sigma_x I$, we have,
\[\tr (K-\lopt)^\top (R+B^\top P_{\lopt}B)(K-\lopt) \Sigma_{K}    \geq \sigma \sigma_x \Vert K-\lopt \Vert_F^2 = \mu \Vert K - \lopt\Vert_F^2. \]
Now we bound the last term in \eqref{eq:global_opt_ineq_1}. Note that inside the expectation in the last term in \eqref{eq:global_opt_ineq_1}, almost surely we have, $\Vert (A-BK)x_t + f(x_t)\Vert = \Vert x_{t+1} \Vert\leq c D_0$, and $\Vert (A-B\lopt)x_t + f(x_t)\Vert \leq (c + \ell) \Vert x_t\Vert \leq 2c^2 D_0 $ (using $\ell\leq 1\leq c$). Therefore, we can invoke the second part of Lemma~\ref{lem:value_func} on the smoothness of $g_{\lopt}$ and get almost surely,
\begin{align*}
    & g_{\lopt}((A-BK)x_t+ f(x_t)) - g_{\lopt}((A-B\lopt)x_t+f(x_t))\\
    &\geq - \tr ( B(K-\lopt) x_t)^\top \nabla g_\lopt((A-BK)x_t + f(x_t)) - \frac{L}{2} \Vert B(K-\lopt) x_t\Vert^2\\
    &\geq -\Vert K - \lopt\Vert_F \Vert B\Vert  \Vert x_t\Vert L\Vert x_{t+1}\Vert -\frac{L}{2} \Vert B\Vert^2 \Vert K - \lopt\Vert_F^2 \Vert x_t\Vert^2  \\
    &\geq - \Vert K - \lopt\Vert_F L \abk  c^2 D_0^2 \rho^{2t}  -\Vert K - \lopt\Vert_F^2 \frac{L\abk^2  c^2 D_0^2}{2} \rho^{2t} .
    \end{align*}
Plugging the above into \eqref{eq:global_opt_ineq_1} and using the easy to check fact that as $K\in\Omega$, $\Vert \Sigma_K^{fx}\Vert_F \leq \E \sum_{t=0}^\infty \ell \Vert x_t\Vert^2 \leq \ell \frac{c^2 D_0^2}{1-\rho}$, we have when $K\in \Omega/\Lambda(\frac{\delta}{3})$,
    \begin{align*}
        C(K) - C(\lopt)&\geq 
     \Big[\mu  -\frac{1}{2} L \frac{\abk^2  c^2 D_0^2}{1-\rho}  \Big] \Vert K - \lopt\Vert_F^2 - \Big[2    \ell \frac{\abk^3 c^4 D_0^2}{(1-\rho)^2}+ L  \frac{\abk  c^2 D_0^2}{1-\rho}  \Big]\Vert K-\lopt\Vert_F\\
     &> \Vert K-\lopt\Vert_F \Big[ \Big(\mu  -\frac{1}{2} L \frac{\abk^2  c^2 D_0^2}{1-\rho}  \Big) \frac{\delta}{3} - 2    \ell \frac{\abk^3 c^4 D_0^2}{(1-\rho)^2}- L  \frac{\abk  c^2 D_0^2}{1-\rho} \Big].
    \end{align*}
    Therefore, it suffices to show that, 
    \begin{align*}
        \frac{1}{2} L \frac{\abk^2  c^2 D_0^2}{1-\rho}=(\ell + 2\ell' c^3D_0) \frac{2\abk^4 c^8 D_0^2}{(1-\rho)^4}  \leq \frac{1}{2}\mu, \nonumber \\
        2    \ell \frac{\abk^3 c^4 D_0^2}{(1-\rho)^2} + L  \frac{\abk  c^2 D_0^2}{1-\rho}  < (\ell + \ell' c^3 D_0) \frac{8 \abk^3 c^8 D_0^2}{(1-\rho)^4}\leq \frac{\delta \mu}{6} .
    \end{align*}
As such, it suffices to require 
{\[\ell \leq \delta \frac{ \sigma \sigma_x (1-\rho)^4}{96 \abk^4 c^8 D_0^2}, \ell' \leq \delta \frac{ \sigma \sigma_x (1-\rho)^4}{96 \abk^4 c^{11} D_0^3} . \]}\qed

\section{Proof of Theorem~\ref{thm:policygradient}: Convergence of Zeroth-Order Policy Search} \label{sec:policy_grad}

We start with the following result regarding the accuracy of the gradient estimator, the proof of which is postponed to Section~\ref{subsec:grad_estimator}. 

\begin{lemma}\label{lem:gradestimator}
Under the conditions of Theorem~\ref{thm:landscape}, when $K\in\Lambda(\frac{2}{3} \delta)$, then given $e_{grad}$, for any $\nu\in(0,1)$, when $r\leq \min(\frac{1}{3}\delta, \frac{1}{3h}e_{grad})$, 
\begin{align*}
J\geq \frac{1}{e_{grad}^2} \frac{d^3}{r^2} \log\frac{4d}{\nu} \max(18(C(K^*) + 2 h \delta^2)^2 , 72 C_{\max}^2), \quad T\geq \frac{2}{1-\sr} \log \frac{6 d C_{\max}}{e_{grad} r},
\end{align*}
where $d = pn$ and $C_{\max}=   \frac{40\abk^2 \cst^2}{1-\sr} D_0^2$, then with probability at least $1-\nu$,
\[\Vert \widehat{\nabla C}(K) - \nabla C(K)\Vert_F \leq e_{grad} .\]
\end{lemma}

With the bound on the gradient estimator, we proceed to the proof of Theorem~\ref{thm:policygradient}. 

\textit{Proof of Theorem~\ref{thm:policygradient}}. Let $\mathcal{F}_m$ be the filtration generated by $\{\widehat{\nabla C}(K_{m'})\}_{m'=0}^{m-1}$. Then, we have $K_{m}$ is $\mathcal{F}_m$ measurable. We define the following event, 
\begin{align*}
    \mathcal{E}_m &= \{ K_{m'}\in \Ball(K^*,\frac{\delta}{3}),\forall m'=0,1,\ldots,m \} \\
    &\qquad \cap \{ \Vert \widehat{\nabla C}(K_{m'}) - \nabla C(K_{m'})\Vert_F \leq e_{grad}  ,\forall m'=0,1,\ldots,m-1 \},
\end{align*}
 where $\Ball(K^*,\frac{\delta}{3}) = \{K: \Vert K - K^*\Vert_F\leq \frac{\delta}{3}\}$, i.e. the ball centered at $K^*$ with radius $\frac{\delta}{3}$. Clearly, $\mathcal{E}_m$ is also $\mathcal{F}_m$-measurable. We now show that conditioned on $\mathcal{E}_m$ is true, $\mathcal{E}_{m+1}$ happens with high probability, or in other words the following inequality, 
\begin{align}
    \mathbb{E}(\mathbf{1}(\mathcal{E}_{m+1})|\mathcal{F}_m ) \mathbf{1}(\mathcal{E}_m) \geq (1 - \frac{\nu}{M}) \mathbf{1}(\mathcal{E}_m). \label{eq:policy_grad_prob_event}
\end{align}
To show \eqref{eq:policy_grad_prob_event}, we now condition on $\mathcal{F}_m$. On event $\mathcal{E}_m$, we have by triangle inequality, $\Vert K_m - \lopt\Vert_F \leq \Vert K^* - \lopt\Vert_F + \frac{\delta}{3} \leq \frac{2}{3}\delta$, and hence $K_m\in\Lambda(\frac{2}{3}\delta)$. Therefore, by Lemma~\ref{lem:gradestimator} and our selection of $r,J,T$, we have $\Vert \widehat{\nabla C}(K_m) - \nabla C(K_m)\Vert_F \leq e_{grad}$ with probability at least $1 - \frac{\nu}{M}$ (note we have replaced $\nu$ with $\nu/M$ in 
Lemma~\ref{lem:gradestimator}), which, as we show now, will further imply $K_{m+1}\in \Ball(K^*,\frac{1}{3}\delta)$. To see this, as $K_m\in\Lambda(\frac{2}{3}\delta)$, we can use the $\mu$-strong convexity and $h$-smoothness to get,
    \begin{align}
        \Vert K_{m+1} - K^*\Vert_F &\leq \Vert K_m - \eta\nabla C(K_m) - K^*\Vert_F +\eta \Vert \widehat{\nabla C}(K_m) - \nabla C(K_m)\Vert_F \nonumber \\
        &\leq (1-\eta \mu)\Vert K_m - K^*\Vert_F + \eta e_{grad} \label{eq:policy_grad_onestep} \\
        &\leq \max(\frac{\delta}{3},\frac{1}{\mu} e_{grad})\leq \frac{\delta}{3}, \nonumber 
    \end{align}
where the second inequality is due to the contraction of gradient descent for strongly convex and smooth functions \citep{bubeck2014convex}, and in the last step, we have used $e_{grad} \leq  \mu\frac{\delta}{3}$. As such, \eqref{eq:policy_grad_prob_event} is true, and taking expectation on both side, we get, 
    \[\mathbb{P}(\mathcal{E}_{m+1}) = \mathbb{P}(\mathcal{E}_{m+1}\cap \mathcal{E}_m) = \E [\mathbb{E}(\mathbf{1}(\mathcal{E}_{m+1})|\mathcal{F}_m ) \mathbf{1}(\mathcal{E}_m)  ]\geq (1 - \frac{\nu}{M}) \mathbb{P}(\mathcal{E}_m). \]
As a result, we have, $\mathbb{P}(\mathcal{E}_{M}) \geq (1 - \frac{\nu}{M})^M \mathbb{P}(\mathcal{E}_0) > 1 - \nu$, where we have used $\mathcal{E}_0$ is true almost surely as $K_0 = \lopt \in \Ball(K^*,\frac{\delta}{3})$. 

Now, on the event $\mathcal{E}_{M}$, we have \eqref{eq:policy_grad_onestep} is true for all $ m=0,\ldots, M-1$. As such, we have,
\begin{align*}
    \Vert K_{M} - K^*\Vert_F &\leq (1 - \eta\mu)^{M}\Vert K_0 - K^*\Vert_F + \eta e_{grad}\sum_{m=0}^{M-1} (1-\eta\mu)^{m}\\
    &\leq (1-\eta\mu)^M \frac{\delta}{3} + \frac{1}{\mu} e_{grad} \\
    &\leq   \sqrt{\frac{2\varepsilon}{h}},
\end{align*}
where we have used $M\geq \frac{1}{\eta\mu} \log ( \delta \sqrt{\frac{h}{\varepsilon}})$, $e_{grad}\leq \frac{\mu}{2}\sqrt{\frac{\varepsilon}{h}}$. 
As such, by $h$-smoothness,
\begin{align*}
    C(K_M)\leq C(K^*) + \frac{h}{2} \Vert K_M - K^*\Vert_F^2\leq C(K^*)  + \varepsilon,
\end{align*}
which is the desired result. Note that the above is true only when conditioned on $\mathcal{E}_M$, as such the desired result is true with probability at least $1-\nu$.
\qed

\subsection{Proof of Lemma~\ref{lem:gradestimator}}\label{subsec:grad_estimator}
\begin{proof}
As $r\leq \frac{1}{3}\delta$ we have $K+U_j \in\Lambda(\delta)$ for all $j$. As such, both $K$ and $K+U_j$ are inside $\Lambda(\delta)$, in which $C$ is $\mu$-strongly convex and $h$-smooth. 

We start with a standard result in zeroth order optimization \citep{nesterov2017random}. Define a ``smoothed'' version of the cost, $C_r(K) = \E_{U\sim \Ball(r)} C(K+U)$, where $\Ball(r)$ is the Ball centered at the origin with radius $r$ (in Frobenius norm). Then by \citet[Lem. 2.1]{zerothorder},
\begin{align}
  \nabla C_r(K) = \frac{d}{r^2} \E_{U\sim \Sphere(r)} C(K+U)U.  \label{eq:gradientestimator:sphere} 
\end{align}
Further, denote $C_j = C(K+U_j) $. With these definitions, we decompose the error in gradient estimation into three terms,
\begin{align}
&\Vert \widehat{\nabla C}(K) - \nabla C(K)\Vert_F \nonumber \\
&\leq   \underbrace{\Vert \nabla C_{r}(K) - \nabla C(K)\Vert_F}_{:=e_1} + \underbrace{\Vert  \frac{1}{J}\sum_{j=1}^J \frac{d}{r^2}C_j  U_j - \nabla C_{r}(K) \Vert_F}_{:=e_2} +\underbrace{\Vert \frac{1}{J}\sum_{j=1}^J \frac{d}{r^2}\widehat{C}_j  U_j - \frac{1}{J}\sum_{j=1}^J \frac{d}{r^2}C_j  U_j\Vert_F}_{:= e_3} .  
\end{align}
In what follows, we show that $e_1 \leq \frac{1}{3} e_{grad}$ almost surely, $e_2  \leq \frac{1}{3} e_{grad} $ with probability at least $1 - \frac{\nu}{2}$, and $e_3 \leq \frac{1}{3} e_{grad} $ with probability at least $1 - \frac{\nu}{2}$. These together will lead to the desired result.

\textbf{Bounding $e_1$.} By the definition of $C_r(\cdot)$, we have $\nabla C_r(K) = \E_{U\sim \Ball(r)} \nabla C(K+U)$. As such, as $\nabla C(\cdot)$ is $h$-Lipschitz, 
\[e_1 = \Vert \nabla C_r(K) - \nabla C(K)\Vert_F\leq   \E_{U\sim \Ball(r)} \Vert \nabla C(K+U) - \nabla C(K)\Vert_F \leq h r \leq \frac{1}{3}e_{grad}, \]
where in the last step, we have used $r\leq \frac{1}{3h} e_{grad}$. 

\textbf{Bounding $e_2$.} For each $j$, $\frac{d}{r^2}C_j  U_j$ is drawn i.i.d. from $\frac{d}{r^2}C(K+U)U$ with $U\sim \Sphere(r)$ and its expectation is $\E C_j = \nabla C_r(K)$ (cf. \eqref{eq:gradientestimator:sphere}). Further, almost surely,
\[\Vert \frac{d}{r^2}C_j  U_j\Vert_F \leq \frac{d}{r} C(K + U_j)\leq \frac{d}{r} \big(C(K^*) + \frac{h}{2}\Vert K+U_j-K^*\Vert_F^2\big)  \leq \frac{d}{r} (C(K^*) + 2 h \delta^2).  \]
As such, using Hoeffding's bound, we have with probability at least $1 - \frac{\nu}{2}$, 
\begin{align}
    e_2 = \Vert  \frac{1}{J}\sum_{j=1}^J \frac{d}{r^2}C_j  U_j - \nabla C_{r}(K) \Vert_F\leq \frac{d^{1.5}}{r} (C(K^*) + 2 h \delta^2) \sqrt{\frac{2}{J} \log \frac{4d}{\nu}} \leq \frac{1}{3} e_{grad},
\end{align}
where we have used  $J\geq \frac{18}{e_{grad}^2} \frac{d^3}{r^2}(C(K^*) + 2 h \delta^2)^2 \log\frac{4d}{\nu} $.

\textbf{Bounding $e_3$.} We now condition on $\{U_j\}_{j=1}^J$ and focus on the randomness in the initial point $x_0$ of the trajectories generated in the gradient estimator. Let $\tilde{C}_j = \E_{K+U_j} \sum_{t=0}^T [x_t^\top Q x_t+ u_t^\top R u_t] $, where the expectation is taken with respect to the initial state and the trajectory is generated using $K+U_j$. We further decompose $e_3$ into,
\begin{align*}
    e_3 &\leq  \underbrace{\frac{d}{r^2} \Big\Vert \frac{1}{J}\sum_{j=1}^J [\widehat{C}_j  U_j -   \tilde{C}_j  U_j ]\Big\Vert_F}_{:= e_4} + \underbrace{\frac{d}{r^2} \Big\Vert  \frac{1}{J} \sum_{j=1}^J [\widehat{C}_j  U_j -   \tilde{C}_j  U_j ]\Big\Vert_F}_{:=e_5}.
\end{align*}
To bound $e_4$, we note that , the expectation of $ \hat{C}_j U_j $ is $\tilde{C}_j U_j$. Further, note that by Theorem~\ref{thm:landscape}(a), we have $\Vert x_t\Vert \leq c \rho^t \Vert x_0\Vert  \leq c\rho^t D_0$, where $c = 2\cst$ and $\rho = \frac{\sr+1}{2}$. As such, 
\begin{align*}
   |\hat{C}_j| = \sum_{t=0}^T [x_t^\top Q x_t+ u_t^\top R u_t]
    &\leq \Vert Q + (K+U_j)^\top R(K+U_j)\Vert \frac{c^2}{1-\rho^2}D_0^2\\
    &\leq  \frac{5\abk^2 c^2}{1-\rho} D_0^2 := C_{\max} . 
\end{align*}
As such, when condition on $\{U_j\}_{j=1}^J$, the summation in $e_4$ is a summation of independent random variables with zero mean and is bounded. As such, we have by Hoeffding bound, with probability at least $1 - \frac{\nu}{2}$,
\begin{align}
    e_4 \leq \frac{d^{1.5}}{r} C_{\max}\sqrt{\frac{2}{J} \log \frac{4d}{\nu}} \leq \frac{1}{6}e_{grad}, \label{eq:gradientestimator:e4}
\end{align}
where we have used that {$J\geq \frac{72 d^3}{r^2 e_{grad}^2} C_{\max}^2 \log\frac{4d}{\nu}$}. Finally, we have, 
\begin{align*}
    |\tilde{C}_j - C_j| &= \Big|\E \sum_{t=T+1}^\infty [x_t^\top Q x_t+ u_t^\top R u_t]\Big| \\
    &\leq \Vert Q + (K+U_j)^\top R(K+U_j)\Vert \frac{c^2}{1-\rho^2}D_0^2 \rho^{T+1}\\
    &\leq C_{\max} \rho^{T+1}. 
\end{align*}
As such, 
\begin{align}
    e_5 \leq \frac{d}{r} C_{\max}\rho^{T+1} \leq \frac{1}{6} e_{grad}, \label{eq:gradientestimator:e5}
\end{align}
where we have used $T\geq \frac{2}{1-\sr} \log \frac{6 d C_{\max}}{e_{grad} r}$. Combining \eqref{eq:gradientestimator:e4} and \eqref{eq:gradientestimator:e5}, we have $e_3 \leq \frac{1}{3} e_{grad} $ with probability at least $1 - \frac{\nu}{2}$. This concludes the proof of Lemma~\ref{lem:gradestimator}. 
\end{proof}

}{}

\end{document}